%% file: main.tex
\documentclass[10pt,a4paper,oneside,reqno]{amsart}

\usepackage{amsmath,geometry,amssymb,mathpartir,mathtools,latexsym,amsthm,enumerate,leftidx,tikz,url}

\usepackage[all]{xy}

\SelectTips{cm}{}
\input{prooftree}

\numberwithin{equation}{section}

\makeatletter
\def\@tocline#1#2#3#4#5#6#7{\relax
\ifnum #1>\c@tocdepth 
  \else 
    \par \addpenalty\@secpenalty\addvspace{#2}%
\begingroup \hyphenpenalty\@M
    \@ifempty{#4}{%
      \@tempdima\csname r@tocindent\number#1\endcsname\relax
 }{%
   \@tempdima#4\relax
 }%
 \parindent\z@ \leftskip#3\relax \advance\leftskip\@tempdima\relax
 \rightskip\@pnumwidth plus4em \parfillskip-\@pnumwidth
 #5\leavevmode\hskip-\@tempdima #6\nobreak\relax
 \ifnum#1<0\hfill\else\dotfill\fi\hbox to\@pnumwidth{\@tocpagenum{#7}}\par
 \nobreak
 \endgroup
  \fi}
\makeatother

\def\noqed{\renewcommand{\qedsymbol}{}}

\newtheoremstyle{mythm}%
{10pt}
{}
{\itshape}
{}
{\bfseries}
{.}
{.5em}
{}%

\newtheoremstyle{mydef}%
{10pt}
{3pt}
{}
{}
{\bfseries}
{.}
{.5em}
{}%

\newtheoremstyle{myrmk}%
{10pt}
{3pt}
{}
{}
{\itshape}
{.}
{.5em}
{}%

\theoremstyle{mythm}
\newtheorem{theorem}{Theorem}[section]
\newtheorem*{theorem*}{Theorem}
\newtheorem{lemma}[theorem]{Lemma} 
\newtheorem{proposition}[theorem]{Proposition} 
\newtheorem{corollary}[theorem]{Corollary}

\theoremstyle{mydef}
\newtheorem{definition}[theorem]{Definition}	
\newtheorem*{definition*}{Definition}	
\theoremstyle{myrmk}
\newtheorem{remark}[theorem]{Remark} 
 
\newtheorem*{remark*}{Remark} 
\newtheorem*{remarks*}{Remarks}

\newtheorem*{example*}{Example}
\newtheorem*{examples*}{Examples}

\newcommand{\ie}{\text{i.e.\ }}

\newcommand{\myemph}[1]{\textit{#1}}
\newcommand{\by}[1]{\quad&&\text{by {$#1$}}}

\newcommand{\deq}{=}

\newcommand{\defeq}{=_{\mathrm{def}}}
\newcommand{\co}{\,{:}\,}
\newcommand{\type}{\mathsf{type}}

\newcommand{\iso}{\cong}

\newcommand{\ct}{\cdot}

\newcommand{\Hint}{\mathcal{H}}
\newcommand{\Hext}{\mathcal{H}^{\mathrm{ext}}}

\newcommand{\hfiber}{\mathsf{hfiber}}
\newcommand{\iscontr}{\mathsf{iscontr}}

\newcommand{\isprop}{\mathsf{isprop}}
\newcommand{\isequiv}{\mathsf{isequiv}}
\newcommand{\isind}{\mathsf{isind}}
\newcommand{\isbipind}{\mathsf{isind}}
\newcommand{\isalgind}{\mathsf{isind}}
\newcommand{\ishinit}{\mathsf{ishinit}}
\newcommand{\isbiphinit}{\mathsf{ishinit}}
\newcommand{\isalghinit}{\mathsf{ishinit}}

\newcommand{\Hot}{\mathsf{Hot}}
\newcommand{\Eq}{\mathsf{Equiv}}

\newcommand{\ext}{\mathsf{ext}}
\renewcommand{\int}{\mathsf{int}}

\newcommand{\Bool}{\mathsf{Bool}}

\newcommand{\boolind}{\mathsf{boolelim}}

\newcommand{\suc}{\mathsf{succ}}

\newcommand{\Id}{\mathsf{Id}}
\newcommand{\id}[1]{\Id_{#1}}
\newcommand{\refl}{\mathsf{refl}}
\newcommand{\idrec}{\mathsf{J}}

\newcommand{\app}{\mathsf{app}}

\newcommand{\pair}{\mathsf{pair}}
\newcommand{\mysplit}{\mathsf{split}}

\newcommand{\W}{\mathrm{W}}
\newcommand{\wsup}{\mathsf{sup}}

\newcommand{\wind}{\mathsf{elim}}

\newcommand{\U}{\mathsf{U}}

\newcommand{\Bip}{\mathsf{Bip}}
\newcommand{\BipHom}{\mathsf{Bip}}
\newcommand{\BipHot}{\mathsf{BipHot}}
\newcommand{\FibBip}{\mathsf{FibBip}}
\newcommand{\BipSec}{\mathsf{BipSec}}
\newcommand{\isbipequiv}{\mathsf{isbipequiv}}
\newcommand{\BipEquiv}{\mathsf{BipEquiv}}

\newcommand{\elim}{\mathsf{elim}}
\newcommand{\comp}{\mathsf{comp}}
\newcommand{\rec}{\mathsf{rec}}

\newcommand{\Palg}{\mathsf{Alg}}
\renewcommand{\sup}{\mathrm{sup}}
\newcommand{\isalgequiv}{\mathsf{isalgequiv}}
\newcommand{\AlgEquiv}{\mathsf{AlgEquiv}}
\newcommand{\AlgHot}{\mathsf{AlgHot}}

\newcommand{\FibPalg}{\mathsf{FibAlg}}
\newcommand{\PalgSec}{\mathsf{AlgSec}}
\newcommand{\AlgSecHot}{\mathsf{AlgSecHot}}

\begin{document}

\title[]{Homotopy-initial algebras in type theory}
\author[S. Awodey]{STEVE AWODEY}
\address{Steve Awodey, Department of Philosophy 
Carnegie Mellon University 
Pittsburgh, PA  15213, USA}
\email{awodey@cmu.edu}
\author[N. Gambino]{NICOLA GAMBINO}
\address{Nicola Gambino, School of Mathematics, University of Leeds, Leeds LS2 9JT, UK}
\email{n.gambino@leeds.ac.uk}
\author[K. Sojakova]{KRISTINA SOJAKOVA}
\address{Kristina Sojakova, Department of Computer Science, Carnegie Mellon University, Pittsburgh, PA 15213, USA}
\email{kristinas@cmu.edu}
\date{\today}

\begin{abstract}
We investigate inductive types in type theory, using the insights provided by homotopy type theory and univalent foundations of mathematics. We do so by introducing the new notion of a homotopy-initial algebra. This notion is defined by a purely type-theoretic contractibility condition which replaces the standard, category-theoretic universal property 
involving the existence and uniqueness of appropriate morphisms. 
Our main result characterises the types that are equivalent to $\W$-types as homotopy-initial algebras. 
\end{abstract}

\maketitle

\newcommand{\Nat}{\mathbb{N}}
\newcommand{\natrec}{\mathsf{rec}}
\newcommand{\natelim}{\mathsf{elim}}

\section*{Introduction}
Inductive types, such as the type of natural numbers and types of well-founded trees,  are one of the fundamental ingredients of dependent type theories, including  Martin-L\"of's type theories~\cite{NordstromB:marltt} and the Calculus of Inductive Constructions~\cite{BertotY:inttpp,CoquandT:inddt}. In the present work, we investigate inductive types using the
insights provided by homotopy type theory~\cite{hott} and univalent foundations of mathematics~\cite{VoevodskyV:unifc}.

As an introduction to the general problem that we will investigate, let us consider the case of the type of natural numbers. Its elimination rule can be seen as the propositions-as-types translation of the familiar induction principle:

\begin{equation}
\label{equ:natind}
\tag{E}
\begin{prooftree}
x \co \Nat \vdash E(x) \co \type \quad
c \co E(0) \quad
x \co \Nat, y \co E(x) \vdash d(x,y) \co E(\suc(x)) 
\justifies
x \co \Nat \vdash  \natelim(x, c,d) \co E(x) \, .
\end{prooftree} \medskip
\end{equation}
As is well-known, the special case of the rule~\eqref{equ:natind} obtained by considering the dependent type in its premiss to be constant
 provides a counterpart of the familiar principle of defintion of a function by recursion:

\begin{equation}
\label{equ:natrec}
\tag{R}
\begin{prooftree}
A  \co \type \quad
c \co A \quad
y \co A \vdash d(y) \co A
\justifies
x \co \Nat \vdash  \natrec(x, c, d) \co A \, .
\end{prooftree} \medskip
\end{equation}
This rule is closely related to Lawvere's notion of a natural number object in a category \cite{Lawvere:NNO}. Indeed, 
it allows us to define a function $f \co \Nat \to A$ such that the following diagram commutes:
\[
\xymatrix@R=0.4cm@C=1.4cm{
 & \Nat \ar[r]^{\suc}  \ar[dd]^{f}  & \Nat \ar[dd]^{f}  \\
 1 \ar[ur]^{0} \ar[dr]_{c} & & \\ 
  & A \ar[r]_{d}  & \; A \, . }
  \]
Within type theory, the commutativity of the diagram is expressed by judgemental equalities
\[
f(0) =  c \co A \, , \quad x \co \Nat \vdash f(\suc(x)) = d(f(x)) \co A \, ,
\]
which can be proved as a special case of the computation rules for~$\Nat$. 

In the notion of a natural number object, however, one not only requires the existence of such a function~$f$, 
but also its uniqueness. Remarkably, within type theory, it is possible to use the elimination
rule~\eqref{equ:natind} to show that such a function $f$ is unique up to a pointwise propositional equality,
\ie that given another function $g \co \Nat \to A$ making the corresponding diagram commute, there are
propositional equalities $\phi_x  \co \Id_A(fx, gx)$ for every $x \co \Nat$. This suggests the possibility of 
 characterizing inductive types, such as the type of natural numbers, by means of standard
category-theoretic universal properties. Unfortunately, this seems to be possible only in the
presence of additional extensionality principles such as the equality reflection rule (which forces 
propositional equality to coincide with judgemental equality)~\cite{DybjerP:repids,GoguenH:inddtw,MoerdijkI:weltc}.  Without these principles, the uniqueness up to pointwise propositional equality of the functions defined by recursion does not seem to
be sufficient to derive the elimination and computation rules for inductive types. Indeed,
the elimination rules imply not only the existence of pointwise propositional equalities,
as above, but also their essential uniqueness, expressed by a system of higher and higher propositional equalities
whose combinatorics are difficult to axiomatize directly. 

The aim of this paper is to solve this problem using ideas inspired by the
recent connections between type theory, homotopy theory
and higher-dimensional category theory~\cite{AwodeyS:homtmi,vandenBergB:typwg,gambino_garner,ssets,LumsdaineP:weaci}, which are at the core of homotopy type theory~\cite{hott} and Voevodsky's univalent foundations of mathematics programme~\cite{VoevodskyV:notts}. 
Our analysis focuses on well-ordering types (W-types for short),
which can be easily characterized as initial algebras for polynomial functors within extensional type
theories~\cite{AbbottM:concsp,DybjerP:repids,GambinoN:weltdp,MoerdijkI:weltc}. Our
results show that in the system under consideration, a type is equivalent to a W-type if and only if
it is a \emph{homotopy-initial algebra} for a polynomial functor. The notion of homotopy-initial
algebra, which we introduce here, is intended as a generalization of the standard category-theoretic
notion of an initial algebra, obtained by replacing the usual existence and uniqueness requirements
by asking for the \emph{contractibility} of suitable types of algebra morphisms. The notion 
of homotopy-initial algebra is entirely type-theoretic, but it is inspired by ideas of higher-dimensional
category theory, where standard category-theoretic universal properties are generalized using the topological
 notion of contractibility~\cite{LurieJ:higtt}. Although we do not consider the semantics of homotopy-initial
 algebras in
 this paper, we expect it to be given by homotopy-invariant versions of initial
 algebras for polynomial functors (cf.~\cite{vandenBergMoerdijk:Wtypes}). 

As part of our development, we also establish several results that do not have counterparts in the extensional setting.  
For  example, we show how the elements of the identity type between two algebra morphisms are essentially
 type-theoretic counterparts of the notion of an algebra 2-cell~\cite{BlackwellR:twodmt}.
This surprising fact provides further evidence for the idea that the rules for identity types encapsulate 
higher-dimensional categorical structure~\cite{vandenBergB:typwg,LumsdaineP:weaci}. We also
analyze the complexity of the types of proofs that a given type is homotopy-initial, showing that it
is a \emph{mere proposition}, \ie a type of homotopy level 1~\cite{VoevodskyV:unifc}. Finally, we show that, 
under the assumption of Voevodsky's univalence axiom, a version of univalence  also holds for
algebras and that such algebras, when they exist, are
essentially unique, \ie unique up to a contractible type of propositional equalities.  It may be noted that, because of the higher-dimensional structure provided by identity types, polynomial functors may acquire further aspects, not present in the extensional setting 
({cf.}~\cite{KockJ:dattsp}).

Our development can be extended without difficulty to other kinds of inductive types, such as coproducts $A + B$ and the natural numbers  $\Nat$.  In fact, in order to illustrate our ideas,  we begin the paper by considering the simpler case
of the type $\Bool$ of Boolean truth values, establishing analogues of the results proved
later for W-types. 

Some of the results presented here were announced in our extended abstract~\cite{wtypes}, and are summarized in the book \cite{hott}. The present paper expands the material outlined there by including not only all of the omitted proofs (which requires the statement of auxiliary lemmas), but also a new, more algebraic treatment of the elimination and computation rules for inductive types, as well as an analysis of the complexity of the type of proofs that a type is homotopy-initial, and an investigation of the further consequences of the univalence axiom.  
\smallskip

\noindent
\emph{Formalization.} All the results in this paper have been formally verified using the
Coq proof assistant. The formalization files, which build on the existing libraries for homotopy type theory and 
univalent foundations of mathematics, are available from the third author's GitHub repository:
\begin{center}
\url{https://github.com/kristinas/hinitiality}
\end{center}

\smallskip

\noindent
\emph{Organization of the paper.} Section~\ref{sec:bac} reviews all of the  preliminaries necessary to
read the paper and introduces the type theory $\Hint$ which will provide the background theory for
our investigations. The rest of the paper is divided in two parts. The first part considers the type $\Bool$.
We begin in Section~\ref{sec:bip} by defining the notions of a bipointed type, bipointed morphism,
fibered bipointed type, bipointed section analyzing homotopies between morphisms and 
sections in terms of identity types. We also discuss the notion of equivalence between bipointed types.
Section~\ref{sec:homibt} introduces the notions of inductive bipointed type and homotopy-initial
bipointed type, so as to arrive at the main results, characterizing $\Bool$ up to equivalence and
exploring consequences of the univalence axiom. The second part, which comprises
Sections~\ref{section:wfiles} and~\ref{sec:homta}, proceeds in parallel with the 
first part, but with algebras for a polynomial functor instead of bipointed types. 
This second part forms the main contribution of the paper, while the first part provides a simpler setting in which to introduce the new concepts and methods of proof.  The formal structure of the two parts is intentionally parallel, in order to guide the reader through the more difficult, second part.

\section{Homotopy-theoretic concepts in type theory}
\label{sec:bac}

\subsection*{Review of type theory.} The type theories considered in this paper are formulated using the 
following four forms of judgement:
\[
A \co \type \, , \quad A \deq B \co \type \, , \quad   a \co A \, , \quad a \deq b \co A \, . 
\]
We refer to the equality relation in these judgements as \emph{judgemental equality}, 
which should be contrasted with the notion of \emph{propositional equality}
defined below. 
Each kind of judgement can also be made relative to a context of variable declarations $\Gamma$, e.g. $\Gamma \vdash A \co \type$. However, when stating deduction rules we may omit the mention
of a context common to premisses and conclusions of the rule, and we
make use of other standard conventions to simplify the exposition.

We begin by introducing a very basic version of Martin-L\"of's type theory, denoted by~$\mathcal{M}$.
This type theory  has rules for the following forms of type:
\[
(\Sigma x \co A) B(x) \, , \quad 
(\Pi x \co A) B(x) \, , \quad
 \Id_A(a,b) \,  ,  \quad
 \U \, . 
 \]
The rules for these types are recalled in 
Tables~\ref{tab:sigmarules},~\ref{tab:pirules},~\ref{tab:idrules} and~\ref{tab:urules}, respectively. The rules
are as in~\cite{NordstromB:marltt}, except that the rules for the type universe $\U$ are stated 
\`a la Russell  for simplicity. As usual, we refer to an element of the form appearing
in the conclusion of an introduction rule as a \emph{canonical element}.

\begin{table}[htb]
\fbox{\begin{minipage}{14cm}
\[
\begin{prooftree}
x \co A \vdash B(x) \co \type
\justifies
\textstyle
(\Sigma x \co A) B(x) \co \type
\end{prooftree}
\qquad \qquad
\begin{prooftree}
a \co A \qquad 
b(a) \co B(a) 
\justifies
\textstyle
\pair(a,b) \co (\Sigma x \co A) B(x)
\end{prooftree}
\]  \bigskip
\[
\begin{prooftree}
\textstyle
z \co (\Sigma x \co A) B(x) \vdash E(z) \co \type \quad
 x \co A, y \co B(x) \vdash  e(x,y) \co  E(\pair(x,y))  
\justifies
\textstyle
z \co (\Sigma x \co A) B(x) \vdash  \mysplit(z,e) \co  E(z)
\end{prooftree}
\]  \bigskip
\[
\begin{prooftree}
\textstyle
z \co (\Sigma x \co A) B(x) \vdash E(z) \co  \type \quad
 x \co  A, y \co B(x) \vdash  e(x,y) \co  E(\pair(x,y))  
 \justifies
x \co A, y \co B(x) \vdash \mysplit(\pair(x,y),e) = e(x,y) \co  E(\pair(x,y)) 
\end{prooftree} \medskip
\]
\end{minipage}} \smallskip
\caption{Rules for $\Sigma$-types.}
\label{tab:sigmarules} 
\end{table}

\begin{table}[htb]
\fbox{
\begin{minipage}{14cm}
\[
\begin{prooftree}
x \co A \vdash B(x) \co \type
\justifies
\textstyle
(\Pi x \co A) B(x) \co \type
\end{prooftree} \qquad \qquad
\begin{prooftree}
x \co A \vdash b(x) \co B(x) 
\justifies
\textstyle
(\lambda x \co A)b(x) \co (\Pi x \co A)B(x)
\end{prooftree}
\] \medskip 
\[
\begin{prooftree}
\textstyle
f \co (\Pi x \co A) B(x) \quad
a \co A 
\justifies
\app(f, a) \co B(a)
\end{prooftree} \qquad \qquad
\begin{prooftree}
x \co A \vdash b(x) \co B(x) 
\justifies
\app( (\lambda x \co A) b(x), a) = b(a) \co B(a) 
\end{prooftree} \medskip 
\]
\end{minipage}} \medskip
\caption{Rules for $\Pi$-types.}
\label{tab:pirules}

 \end{table}

\begin{table}[htb]
\fbox{
\begin{minipage}{14cm}
\[
\begin{prooftree}
A \co  \type \quad 
a \co  A  \quad
b \co  A 
\justifies
 \id{A}(a,b) \co  \type
 \end{prooftree} \qquad \qquad 
\begin{prooftree}
a \co A 
\justifies
 \refl(a) \co  \id{A}(a,a)
 \end{prooftree} 
\] \bigskip
\[
\begin{prooftree}
x, y \co A, u \co  \id{A}(x,y) \vdash E(x,y,u) \co \type \qquad
 x \co A \vdash  e(x) \co  E(x,x,\refl(x))  
\justifies
x, y \co A, u \co  \id{A}(x,y) \vdash  \idrec(x,y,u,e) \co E(x,y,u)
\end{prooftree}
\] \bigskip
\[
\begin{prooftree}
x, y \co  A, u \co  \id{A}(x,y) \vdash E(x,y,u) \co \type \qquad
 x \co  A \vdash  e(x) \co  E(x,x,\refl(x)) 
 \justifies
x \co A \vdash \idrec(x,x,\refl(x), e) = e(x) \co  E(x, x, \refl(x)) 
\end{prooftree}
\] \smallskip
\end{minipage}} \smallskip
\caption{Rules for $\Id$-types.} 
\label{tab:idrules}
\end{table}

\begin{table}
\fbox{
\begin{minipage}{14cm}
\[
\begin{prooftree}
A \co \U \quad
x \co A \vdash B(x) \co \U 
\justifies
(\Sigma x \co A) B(x) \co \U
\end{prooftree} \qquad
\begin{prooftree}
A \co \U \quad
x \co A \vdash B(x) \co \U 
\justifies
(\Pi x \co A) B(x) \co \U
\end{prooftree}  
\]
\medskip
\[
\begin{prooftree}
A \co \U \quad
a \co A \quad
b \co A 
\justifies
\Id_A(a,b) \co \U
\end{prooftree}  \qquad
\begin{prooftree}
A \co \U 
\justifies
A \co \type
\end{prooftree}
\]
\smallskip
\end{minipage}} \smallskip 
\caption{Rules for the type universe $\U$.}
\label{tab:urules}
\end{table}

Let us establish some notation and recall some
basic facts and  terminogy. First of all, for~$f \co (\Pi x \co A) B(x)$ and $a \co A$, we write~$f(a)$ or~$f  a$ instead of $\app(f,a)$. We may also write~$(a,b)$ instead of $\pair(a,b)$ to denote canonical elements of $\Sigma$-types.
Given types $A$ and~$B$, the product type $A \times B$ and the function $A \rightarrow B$  are defined via $\Sigma$-types and $\Pi$-types in the usual way.  As is standard, we let 
$
A \leftrightarrow B \defeq (A\to B) \times (B\to A)$.
The rules for $\Sigma$-types allow us to derive the rules for projections

\[
\begin{prooftree}
c \co (\Sigma x \co A)B(x) 
\justifies
\pi_1(c) \co A 
\end{prooftree} \qquad
\begin{prooftree}
c \co (\Sigma x \co A)B(x) 
\justifies
\pi_2(c) \co B(\pi_1(c)) \, .
\end{prooftree} \medskip
\]

We say that two elements  $a, b \co A$ are \emph{propositionally equal} if  the type $\Id_A(a,b)$ is inhabited and
write  $a \iso b$ to denote this situation. The rules for $\Sigma$-types allow us to prove the following propositional form of the $\eta$-rule for $\Sigma$-types:
\begin{equation}
\label{equ:etasigma}
\begin{prooftree}
c \co (\Sigma x \co A) B(x)
\justifies
\eta_c \co \Id( c, \pair(\pi_1(c), \pi_2(c))) \, .
\end{prooftree}
\end{equation}
This rule asserts that every element of a $\Sigma$-type is propositionally equal to one of canonical form.
Note that neither $\mathcal{M}$ nor $\Hint$ include the judgemental form of the $\eta$-rules for $\Sigma$-types, 
as is done in~\cite{GoguenH:inddtw}.  The presence of the type universe $\U$ allows us to define the notion of a small type:
as usual, we say that a type  $A$ is \emph{small} if it is an element of  the type universe, \ie $A\co\U$. 

We write $\mathcal{M}^{\mathrm{ext}}$ for the extensional type
theory obtained from $\mathcal{M}$ by adding the following rule, known as the \emph{identity reflection rule}:
\begin{equation}
\label{equ:collapse}
\begin{prooftree}
 p \co  \Id_A(a,b)
  \justifies
  a=b \co  A \, . 
  \end{prooftree} \medskip
\end{equation}
This rule collapses propositional equality to definitional equality, thus making the overall system
somewhat simpler to work with, but makes type-checking undecidable~\cite{HofmannM:extcit}. For this reason, it is not assumed
in the most recent formulations of Martin-L\"of type theories~\cite{NordstromB:marltt} or in automated proof assistants like Coq~\cite{BertotY:inttpp}.  Rather than working in $\mathcal{M}^{\mathrm{ext}}$, we work in a weaker extension of $\mathcal{M}$ which we now describe.

 \subsection*{The type theory $\Hint$}
The type theory $\Hint$ which will serve as the background theory for our development extends the type
theory $\mathcal{M}$ described above with two additional rules. The first additional rule is a judgemental form of the 
$\eta$-rule for $\Pi$-types:
\begin{equation}
\label{equ:etapi}
\begin{prooftree}
\textstyle
f \co (\Pi x \co A) B(x) 
\justifies
\textstyle
f = (\lambda x \co A) \app(f, x) \co  (\Pi x \co A)B(x) \, .
\end{prooftree} 
\end{equation}
An immediate consequence of this rule is that we can identify a family of small types,
given by a dependent type $x \co A \vdash B(x) \co \U$ with functions $B \co A \to \U$.
In the followiing, we shall refer to both of these as \emph{small dependent types}.  
The second additional rule is the \emph{function extensionality} axiom, 
which is considered here with propositional equalities:
 \begin{equation}
 \label{equ:funext}
 \begin{prooftree} 
 f \co (\Pi x \co A)B(x) \qquad
 g \co (\Pi x \co A) B(x) \qquad
 x \co A \vdash \alpha_x \co \Id_{B(x)}(f(x), g(x))
 \justifies
 \mathsf{funext}(f, g, \alpha) \co \Id_{(\Pi x \co A)B(x)}(f,g) \, .
 \end{prooftree}
\end{equation}
 
\medskip

Note that $\Hint$ does not have any ground types apart from the type universe $\U$. This is because these type theories are intended as background
theories for our study of inductive types. The type theory~$\Hint$ 
does not include any global extensionality principles, like the identity reflection rule, the K rule, or 
the uniqueness of identity proofs (UIP) principle~\cite{StreicherT:invitt}. This makes it possible
for~$\Hint$ to have not only straightforward set-theoretic models (where those extensionality
principles are valid), but also with homotopy-theoretic models, such as the groupoid model~\cite{HofmannM:gromtt}
and the simplicial model~\cite{ssets}, in which the rules of~$\Hint$, but not the extensionality principles mentioned above, remain valid. Indeed, $\Hint$ is a subsystem of the type theory 
used in Voevodsky's univalent foundations of Mathematics programme~\cite{VoevodskyV:unifc}. 
In particular, the 
function extensionality axiom in~\eqref{equ:funext} is formally implied by the univalence axiom~\cite{VoevodskyV:notts} (using the fact that function extensionality, as stated in~\eqref{equ:funext}, follows from its special case for function types). But, in contrast with the univalence
axiom, the function extensionality axiom is valid also in set-theoretic models. Uses of the univalence axiom will be explicitely noted.

We write~$\Hext$ for the  extension of~$\Hint$ with the identity reflection rule in~\eqref{equ:collapse}.

\begin{remark*} Our results continue to hold when the judgemental $\eta$-rule for $\Pi$-types in~\eqref{equ:etapi}
is weakened by replacing the judgemental equality in its conclusion with a propositional one, which is
derivable if $\Pi$-types are defined as inductive types, as done in~\cite{NordstromB:promlt}. However, since
some of our proofs can be  simplified in its presence and the
current version of the Coq proof assistant assumes the rule~\eqref{equ:etapi}, 
we prefer to work with it in order to keep our presentation simpler and 
more faithful to the formalization.
\end{remark*}

\subsection*{Homotopy-theoretic notions in type theory} For the convenience of the reader, 
we review some ideas developed in more detail in~\cite{hott,VoevodskyV:notts}. 
First of all, we will frequently refer to elements of identity types of the form 
$p \co \Id_A(a,b)$ as \emph{paths} (from $a$ to $b$ in $A$). By the $\Id$-elimination rules, 
for every dependent type
\begin{equation}
\label{equ:dependenttype}
x\co A \vdash E(x) \co \type \, ,
\end{equation} 
a path $p\co \Id_A(a,b)$ determines the so-called \emph{transport functions} 
\[
p_{\, ! } \co E(a) \rightarrow E(b) \, , \quad p^* \co E(b) \to E(a) \, .
\] 
These are defined so that, for $x \co A$, the functions $\refl(x)_{\, !}$ 
and~$\refl(x)^*$ are the identity function $1_{E(x)} \co E(x) \to E(x)$.  In order to emphasize the fact
that dependent types are interpreted as fibrations in homotopy-theoretic models, we sometimes refer to a dependent
type as in~\eqref{equ:dependenttype} as a \emph{fibered type} over $A$. Accordingly, elements of the 
type~$(\Pi x \co A)E(x)$ may be referred to as~\emph{sections} of the fibered type. This terminology is supported by the fact that a section $f \co (\Pi x \co A)E(x)$ determines a function $f' \co A \to E'$, where $E' \defeq (\Sigma x \co A)E(x)$, defined by
$f' \defeq (\lambda x \co A) \pair(x, fx)$, which is such that $ \pi_1 f'(x) = x$ for every $x \co A$.  We represent such a situation with the diagram
  \[
   \xymatrix{
    E' \ar[d]_{\pi_1} \\
    \;  A \, . \ar@/_1pc/[u]_{f'} }
     \]

Let us now review the notion of an equivalence of types. In order to do this, we need some auxiliary notions. Recall that a type $A$ is said to be \emph{contractible} if the  type 
 \begin{equation}
 \label{eq:contractible}
\iscontr(A) \defeq (\Sigma x\co A) (\Pi y\co A) \Id_A(x,y)
\end{equation}
is inhabited. The type $\iscontr(A)$ can be seen as the propositions-as-types translation
of the formula stating that $A$ has a unique element. However, its homotopical interpretation 
is as a space that is inhabited if and only if the space interpreting $A$ is contractible in the usual
topological sense. Next, we define the \emph{homotopy fiber} of a function $f \co A \to B$ over $y \co B$ as the type
\[
 \hfiber(f,y) \defeq (\Sigma x \co A)\, \Id_B(f x, y) \, .
\]
A function $f \co A \to B$ is then  said to be an \emph{equivalence} if and only if all of its homotopy fibers are contractible, \ie the type
\[
\isequiv(f) \defeq (\Pi y \co B) \, \iscontr(\hfiber(f,y)) 
\]
is inhabited.  For types $A$ and $B$, the type $\Eq(A,B)$ of equivalences from $A$ to $B$ is  defined so that its canonical elements are pairs consisting of a function $f \co A \to B$ and a proof that it is an equivalence, \ie we let 
 \begin{equation}
 \label{equ:weq}
 \Eq(A,B) \defeq (\Sigma f \co A \to B) \, \isequiv(f) \, .
 \end{equation}
 We write $A\simeq B$ if there is an equivalence from $A$ to $B$. For example, the well-known $\Pi\Sigma$-distributivity, which is sometimes referred to as the \emph{type-theoretic axiom of choice}~\cite{MartinLofP:inttt}, can be expressed as an equivalence 
 \begin{equation}
 \label{equ:ac}
 (\Pi x \co A)(\Sigma y \co B(x)) E(x,y) \simeq 
 (\Sigma u  \co (\Pi x \co A)B(x)) (\Pi x \co A) E(x, ux) \, .
 \end{equation}
 It can be shown that a function $f \co A \rightarrow B$ is an equivalence if and only if it has a two-sided inverse,
 \ie there exists a function $g \co B \to A$ such that the types $\Id(gf, 1_A)$
 and $\Id(fg, 1_B)$ are inhabited. However, the type of equivalences is not equivalent to
 the type of functions with a two-sided inverse as above, but instead (as suggested by Andr\'e Joyal) 
to the type of functions that have a left inverse and a right inverse, \ie functions
$g \co B \to A$ and $h \co B \to A$ such that the  types $ \Id(g  f, 1_A)$ and $\Id(f h, 1_B)$ are inhabited. 
More precisely, for every $f \co A \to B$, there is an equivalence
\begin{equation}
\label{equ:equivalternative}
\isequiv(f) \simeq 
 \big( 
 (\Sigma g \co B \to A) \Id(gf, 1_A) \times (\Sigma h \co B \to A) \Id(fh, 1_B) 
 \big) \, .
 \end{equation}
For our purposes, the idea of equivalences as functions with a left and a right inverse will be most
easily generalized when we consider types equipped with additional structure.

Because of the presence of the principle of function extensionality in $\Hint$, 
identity types of function types and of $\Pi$-types admit an equivalent description in terms of the notion of
a homotopy, which we now review.  For $f \, , g \co (\Pi x \co A) B(x)$, the type of homotopies between $f$ and $g$ 
is defined by letting
\[
\Hot(f,g) \defeq (\Pi x:A) \Id_{B(x)}(fx,gx) \, .
\]
We sometimes write  $\alpha \co f \sim g$ rather than $\alpha \co \Hot(f,g)$. 

One of the key insights derived from the homotopy-theoretic interpretation of type theories is that
the notion of contractibility in~\eqref{eq:contractible}  can be used to articulate the world of types  
into a hierarchy of so-called \emph{homotopy levels} (or h-levels for short) according to their homotopical complexity \cite{VoevodskyV:notts}. These
are defined inductively by saying that a type $A$ has level $0$ if it is contractible and it has level
$n+1$ if for every $x, y \co A$ the type $\Id_A(x,y)$ has level $n$. Types of h-level 1 are called 
here \emph{mere propositions}. By definition, a type $A$ is said to be a \emph{mere proposition} if the type
\[
\isprop(A) \defeq (\Pi x \co A)(\Pi  y \co A) \, \iscontr( \Id_A(x,y)) 
\]
is inhabited.

\subsection*{Characterization of identity types} \label{sec:chait} We now recall that the identity types of various kinds of 
compound types admit an equivalent description. We begin by considering product types and function types.
Let $A$ and $B$ be types. For any  $c, d  \co A \times B$, and any $f, g \co A \to B$, we have  canonical maps
\begin{align*} 
\ext^{\times}_{c,d}  \co & \Id_{A \times B}(c, d) \to \Id_{A}(\pi_1 c, \pi_1 d) \times \Id_{B}(\pi_2 c, \pi_2 d) \, , \\
\ext^{\to}_{f,g} \co & \Id_{A \to B}(f, g) \to (\Pi x \co A) \Id_B(fx, gx) \, .
\intertext{Note that the codomain of the second map is $\Hot(f,g)$. These functions can be easily generalized to $\Sigma$-types and $\Pi$-types, so as to obtain functions}
\ext^\Sigma_{c,d} \co &  \Id_{(\Sigma x \co A)B(x)}(c, d) \to (\Sigma p \co \Id_A( \pi_1 c, \pi_1 d)) \,  \Id_{B(\pi_2 d)} ( p_{!}( \pi_2 c), \pi_2 d) \, , \\ 
\ext^\Pi_{f,g} \co & \Id_{(\Pi x \co A)B(x)}(f, g) \to (\Pi x \co A) \Id_{B(x)}(f x, gx)  \, . \\ 
\intertext{Again, the codomain of the second map is $\Hot(f,g)$. Furthermore, for the type universe $\U$, there is an evident function}
\ext^\U_{A,B} \co & \Id_\U(A,B) \to \Eq(A,B) \, .
\end{align*}
We refer to these functions as the \emph{extension functions} for product types, function types, $\Sigma$-types,
$\Pi$-types and $\U$, respectively.  We then have that the extension functions for product types and $\Sigma$-types 
can be shown to be equivalences within the type theory~$\mathcal{M}$, using the (provable) $\eta$-rule for $\Sigma$-types in~\eqref{equ:etasigma}.
The extension functions for function types and $\Pi$-types, for their part, can be shown to be equivalences within the type theory $\Hint$, using the function extensionality principle in~\eqref{equ:funext} that is part of $\Hint$. Finally, the assertion that the extension function for the type universe is an equivalence is exactly the univalence axiom. Thus, within the type theory $\Hint$ we have the following inverses to the
extension functions
\begin{align*}
\int^{\times}_{c,d} \co & \big( \Id_A(\pi_1 c \, , \pi_1 d) \times \Id_B(\pi_2 c,\pi_2 d) \big) \to 
\Id_{A \times B}(c,d)   \\
\int^{\to}_{f,g} \co  & \big( (\Pi x \co A) \Id_B (fx, gx) \big)) \to \Id_{A \to B}(f, g) \\ 
\int^{\Sigma}_{c,d} \co & \big( (\Sigma p \co \Id_A( \pi_1 c, \pi_1 d )) \Id_{B(\pi_2 c)}( p_{!} \pi_2 c, 
\pi_2 d) \big) \to \Id_{(\Sigma x \co A)B(x)}(c, d)  \\
\int^{\Pi}_{f,g} \co & (\Pi x \co A) \Id_{B(x)}(f x, gx) \to  \Id_{(\Pi x \co A)B(x)}(f, g)  \, ,\\ 
 \intertext{and, in the extension of $\Hint$ with the univalence axiom, also the inverse}  
\int^{\U}_{A,B} \co & \Id_\U(A, B) \to  \Eq(A, B)   \, . 
 \end{align*}
In the following, if the context does not create any confusion, we may omit  superscripts and subscripts when manipulating these functions, writing simply $\ext$ and $\int$. Let us also remark that for $\Sigma$-types we could have also used $p^*$ instead
of $p_{!}$, making the evident changes. In the following, we shall use both, depending on which is more convenient.

\subsection*{Higher-dimensional categorical structure} Even if our development is entirely syntactic,
many of the ideas presented in the paper are inspired by concepts of homotopy theory and higher-dimensional
algebra. Therefore, we conclude this preliminary section by  discussing some aspects of the relationship with higher-dimensional category theory, so as to provide further insight into our development.

 First of all, observe that types and functions can be organized into an ordinary 
 category, where the composition and identity laws hold as judgemental equalities. Indeed, if we define the composite $g \circ f \co A 
\to C$ of $f \co A \to B$ and $g \co B \to C$ by letting
\[
 g \circ f \defeq (\lambda x \co A) g ( f  x) \, ,
 \]
 and the identity $1_A \co A \to A$ by letting $1_A \defeq (\lambda x \co A) x$, 
 the presence of the judgemental $\eta$-rule for $\Pi$-types in~\eqref{equ:etapi} in $\Hint$ implies that
 we have judgemental equalities
 \begin{equation}
 \label{equ:assoc}
  h \circ (g \circ f) = (h \circ g) \circ f \, , \quad 1_B \circ f =  f \, , \quad  f \circ 1_A = f \, .
  \end{equation}
  Because of the strict associativity, we may omit bracketing of multiple composites and sometimes write simply $gf$ instead of $g \circ f$. 

The presence of identity types in our type theories, however, equips this category with additional structure. Each  
type~$A$ is a weak $\infty$-groupoid, having  elements of~$A$ as objects, paths
$p \co \Id_A(a,b)$ as 1-morphisms (from~$a$ to~$b$) and elements of iterated identity types as 
$n$-morphisms~\cite{LumsdaineP:weaci,vandenBergB:typwg}. 
We may write 
\[
q \ct p \co \Id_A(a,c) \, , \quad
1_a \co \Id_A(a,a) \, , \quad 
p^{-1} \co \Id_A(b,a) \, ,
\]
for the path obtained by composing $p \co \Id_A(a,b)$ and~$q \co \Id_A(a,c)$,  for the path~$\refl(a) \co \Id_A(a,a)$, and for the quasi-inverse of $p \co \Id_A(a,b)$, respectively~\cite{HofmannM:gromtt}. 
When manipulating this structure, we  refer to the  propositional equalities holding between 
various composites as the \emph{groupoid laws}.

The category of types and functions can then be considered informally as enriched in $\infty$-groupoids
(and hence as an $(\infty, 1)$-category\footnote{We follow convention of using $(\infty, n)$-category to denote an $\infty$-category in which $k$-morphisms, for $k > n$, are invertible. An $\infty$-groupoid is then the same thing as an $(\infty,0)$-category.}), since function types  $A \to B$, just like any other type,
are $\infty$-groupoids. This $(\infty, 1)$-category has types as objects,
functions as 1-morphisms, paths $p \co \Id_{A \to B}(f, g)$ as 2-morphisms, and higher paths as $n$-morphisms. We will not need all the structure of this higher-dimensional category (for which see \cite{Lumsdaine:higcft}), but only some low-dimensional layers of it which can be defined easily.  For example, given functions $f \co A \to B$, $g_1, g_2 \co B \to C$ and a path~$p \co \Id_{B \to C}(g_1, g_2)$, represented diagrammatically as 
\[
\xymatrix{
A \ar[r]^{f} & B \ar@/^1pc/[r]^{g_1} \ar@/_1pc/[r]_{g_2} \ar@{}[r]|{\Downarrow \, p}  & C \, ,}
\]
it is possible to define a path~$p \circ f \co \Id_{A \to C} (g_1 \circ f,  g_2 \circ f)$. 

Because of the equivalences $\Id_{A \to B}(f,g) \simeq \Hot(f,g)$ recalled above, this $(\infty,1)$-category can be described equivalently
as having  types as objects, functions as 1-morphisms, homotopies $\alpha \co \Hot(f, g)$ as 2-morphisms, and higher homotopies as $n$-morphisms. For example, given functions $f \co A \to B$, $g_1, g_2 \co B \to C$ and a homotopy $\alpha \co \Id(g_1, g_2)$, there is a homotopy~$\alpha \circ f \co \Hot(g_1 \circ f, g_2 \circ f)$ which is defined so that, for every $p \co \Id_{B \to C}(g_1, g_2)$, the homotopies~$\ext( p \circ f )$ and~$\ext(p) \circ f$ are propositionally equal, where $\ext$ denotes the extension function for function types.

\section{Bipointed types}
\label{sec:bip}

\subsection*{Bipointed types and bipointed morphisms} \label{sec:biptm}
In this section and the next, we focus on the type $\Bool$ of Boolean truth values. Our development in
these sections provides a template for what we will do for $\W$-types in Section~\ref{section:wfiles}
and Section~\ref{sec:homta} and allows us to present the key ideas in a simpler context. 

The rules for the type $\Bool$ that we consider here are given in Table~\ref{tab:boolrules}. 
The introduction rules state that we have two canonical elements in $\Bool$, written~$0$ and~$1$ here. The elimination rule can be understood as the propositions-as-types translation of an induction principle for $\Bool$. Finally, the computation rules specify what happens if one applies the elimination rule immediately after applying the introduction rule.

\begin{table}[htb]
\fbox{
\begin{minipage}{14cm}
\[
 \Bool \co \type  \qquad \qquad 
0 \co \Bool  \qquad  1 \co \Bool  \qquad \qquad \Bool \co \U 
\]  
\smallskip
\[
\begin{prooftree}
x\in\Bool \vdash E(x) \co \type \qquad
e_0 \co E(0) \qquad
e_1 \co E(1) 
\justifies
x \co \Bool \vdash \boolind(x, e_0, e_1) \co E(x) 
\end{prooftree}
\] \smallskip

\begin{equation*}
\begin{prooftree}
x\in\Bool \vdash E(x) \co \type \qquad
e_0 \co E(0) \qquad
e_1 \co E(1)
\justifies
  \boolind(0, e_0, e_1)  \deq  e_0 \co E(0) \, , 
\end{prooftree}
 \end{equation*}  
 \bigskip
 \begin{equation*}
\begin{prooftree}
x\in\Bool \vdash E(x) \co \type \qquad
e_0 \co E(0) \qquad
e_1 \co E(1)
\justifies
 \boolind(1,e_0,e_1)  \deq e_1 \co E(1) 
\end{prooftree} \bigskip
 \end{equation*}
 \medskip
 \end{minipage}} \smallskip
 \caption{Rules for the type of Boolean truth values.}
 \label{tab:boolrules}
 \end{table}
 
Let us now suppose that we have a small type $A \co \U$ and an equivalence $f \co \Bool \to A$. Then, the type~$A$ 
has two distinguished elements $a_0 \defeq f(0)$ and $a_1 \defeq f(1)$, and it satisfies
analogues of the elimination and computation rules for $\Bool$, except that the conclusions of the computation 
rules need to be modified  by replacing the judgemental equalities  with propositional ones. Our aim in this section is to provide a characterisation of the small types equivalent to $\Bool$ by means of a type-theoretical universal property. But in our development we do not need to assume to have the type $\Bool$, and rather work in the type 
theory~$\Hint$ specified in Section~\ref{sec:bac}.  We begin by introducing the notion of a bipointed type.

\begin{definition} \label{thm:bipointedtype}
A \emph{bipointed type} $(A, a_0, a_1)$ is a type $A$ equipped with two elements  $a_0 \, , a_1 \co A$. 
\end{definition}

When referring to a bipointed type we sometimes suppress mention of its distinguished elements and write $A = (A, a_0, a_1)$ to recall this abuse of language. Similar conventions will be used throughout the paper for other kinds of
structures. In the following, it will be convenient to represent a bipointed type~$A$  
diagrammatically as follows:
\[
\xymatrix{
1 \ar[r]^-{a_0}&  A & 1 \ar[l]_-{a_1} \, .}
 \]
Here, the symbol $1$ does not represent the unit type, which is not assumed as part the type theory $\Hint$.
The type $\Bool$ and its canonical elements $0, 1 \co \Bool$ give us a bipointed type:
\[
\xymatrix{
 1 \ar[r]^-{0}&  \Bool  & 1 \ar[l]_-{1} \, . }
 \]
We say that a bipointed type $A = (A, a_0, a_1)$ is \emph{small} if the type $A$ is a small type, \ie $A \co \U$. 
Accordingly, the type of small bipointed types (which is not small) is then defined by letting 
\[
\Bip \defeq (\Sigma A \co \U)( A \times A ) \, .
\]
Next, we introduce the notion of a bipointed morphism between bipointed types. 
As one might imagine, a bipointed morphism consists of a function between the underlying types which preserves
the bipointed structure. In our context, we formalize this by requiring the existence of appropriate paths, witnessing  
the preservation of structure, as the next definition makes precise. 

\medskip

Let us fix two bipointed types $A = (A, a_0, a_1)$ and~$B = (B, b_0, b_1)$.

\begin{definition} A \emph{bipointed morphism} $(f, \bar{f}_0, \bar{f}_1)  \co A  \to B$
is a function $f \co A \to B$ equipped with paths $\bar{f}_0 \co  \Id(f a_0, b_0)$ 
and~$\bar{f}_1 \co \Id(f a_1, b_1)$.  
\end{definition}

Diagrammatically, we represent a bipointed morphism as follows:
\begin{equation}
\label{equ:bipmor}
\vcenter{\hbox{\xymatrix@C=2cm{
1 \ar[r]   \ar[r]^{a_0} \ar@{=}[d]  \ar@{}[dr]|{\Downarrow \, \bar{f}_0} & A  \ar[d]^{f} & 1  \ar[l]_{a_1} \ar@{=}[d] \ar@{}[dl]|{\Downarrow \,  \bar{f}_1} \\
1 \ar[r]_{b_0}  & B   & \; 1 \, . \ar[l]^{b_1} }}}
 \end{equation}
The type of bipointed morphisms from $A$ to $B$ is then defined by letting
\[
\BipHom(A,B) \defeq (\Sigma f \co A \to B) \big( \Id(  f a_0, b_0 )  \times \Id(  f a_1 , b_1 )  \big) \, .
\]
 Bipointed types and their morphisms behave much like objects and morphisms in a category.
Given two bipointed morphisms  $(f, \bar{f}_0, \bar{f}_1) \co A \to B$ and $(g, \bar{g}_0, \bar{g}_1) \co B \to C$, we can define their composite 
 as the triple consisting of the composite $g  f \co A \to C$ and the paths represented
by the following pasting diagram:
\[
\xymatrix@C=2cm@R=1.2cm{
1 \ar[r]   \ar[r]^{a_0} \ar@{=}[d]  \ar@{}[dr]|{\Downarrow \, \bar{f}_0} & A  \ar[d]^{f} & 1  \ar[l]_{a_1} \ar@{=}[d] \ar@{}[dl]|{\Downarrow \,  \bar{f}_1} \\
1 \ar[r]_{b_0}   \ar@{}[dr]|{\Downarrow \, \bar{g}_0}  \ar@{=}[d] & B \ar[d]^f   & 1 \ar[l]^{b_1} \ar@{=}[d] \ar@{}[dl]|{\Downarrow \,  \bar{g}_1} \\
1 \ar[r]_{c_0}  & C   & \; 1 \, .\ar[l]^{c_1}}
 \]
Explicitly, for $k \in \{ 0, 1 \}$, the path $\overline{(g  f)}_k \co \Id( g f a_k, c_k)$ is obtained as the composite
\[
\xymatrix@C=1.5cm{
g f a_k \ar[r]^{g \circ \bar{f}_k} & g b_k \ar[r]^{\bar{g}_k} & c_k \, .}
\]
Also, for any bipointed type $A = (A, a_0, a_1)$, the identity function $1_A \co A \to A$ can be equipped with the structure of a bipointed by taking $\overline{(1_A)}_k \co \Id( 1_A(a_k), a_k)$ to be $1_{a_k} = \refl(a_k) \co \Id(a_k, a_k)$ for $k \in \{ 0, 1\}$. We represent this as the
diagram
 \begin{equation}
 \label{equ:bipidA}
{\vcenter{\hbox{\xymatrix@C=2cm{
1 \ar[r]^{a_0} \ar@{=}[d] \ar@{}[dr]|{\Downarrow \, 1_{a_0}}& A \ar[d]^{1_A} & 
1 \ar[l]_{a_1} \ar@{=}[d]  \ar@{}[dl]|{\Downarrow \, 1_{a_1}}  \\ 
 1 \ar[r]_{a_0} & A & \; 1 .\ar[l]^{a_1} }}}}
 \end{equation}
 Note that, even if associativity and unit laws for composition for functions between types hold strictly
 (\ie up to judgemental equality, {cf.}~\eqref{equ:assoc}), the  associativity and unit laws for bipointed morphisms do not. This
 is due to the presence of paths in their definition, in complete analogy with the well-known situation
 in homotopy theory~\cite{BoardmanM:homiast}.

We have seen in Section~\ref{sec:chait} that for types~$A$ and~$B$, the identity type of the function
 type~$A \to B$ can be described equivalently as the type of homotopies between functions from $A$ to $B$. As
we show next, it is possible to extend this equivalence to bipointed morphisms. In order to do so, the next definition  introduces the notion of a bipointed homotopy.

\medskip

Let us now fix two bipointed morphisms $f = (f, \bar{f}_0, \bar{f}_1)$ and $g = (g, \bar{g}_0, \bar{g}_1)$ from~$A$
to~$B$.

\begin{definition} \label{thm:biphomotopy} A \emph{bipointed homotopy} 
$(\alpha, \bar{\alpha}_0, \bar{\alpha}_1) \co f \to  g$
is a homotopy~$\alpha \co  \Hot(f, g)$ equipped with paths
$\bar{\alpha}_0 \co \Id(  \bar{f}_0 ,  \bar{g}_0 \ct \alpha_{a_0}  )$ and $\bar{\alpha}_1 \co \Id(
\bar{f}_1 , \bar{g}_1 \ct  \alpha_{a_1})$. 
\end{definition}

Diagrammatically, we represent the paths  involved in a bipointed homotopy as follows:
\[
\xymatrix@C=1.5cm{
f a_k  \ar[r]^{\alpha_{a_k}}  \ar@/_1pc/[dr]_{\bar{f}_k}  
\ar@{}[dr]|{\qquad \Rightarrow \; \bar{\alpha}_k}  & g a_k \ar[d]^{\bar{g}_k}  \\ 
 & \; b_k  \, ,}
  \] 
  for $k \in \{ 0, 1 \}$. The type of bipointed homotopies between $f$ and $g$ is then defined by letting
\[
 \BipHot  \big( (f,\bar{f}_0, \bar{f}_1), (g, \bar{g}_0, \bar{g}_1) \big)   \defeq   
 (\Sigma \alpha \co \Hot( f , g)) \big( 
  \Id\big( \bar{f}_0 ,   \bar{g}_0 \ct \alpha_{a_0}   \big) \times 
  \Id \big( \bar{f}_1,    \bar{g}_1 \ct \alpha_{a_1}  \big) \big) \, .
\]
Lemma~\ref{BoolHomSpace} essentially says that paths between bipointed morphisms
essentially the same thing as bipointed homotopies. This  is
the first instance of the suprising phenomenon, mentioned in the introduction,
that identity types capture higher-dimensional algebraic structures in an apparently
automatic way. It should also be pointed out that, as a consequence of the lemma, types of bipointed homotopies satisfy
analogues of the rules for identity types.

\begin{lemma} \label{BoolHomSpace} 
The canonical  function 
\[
\ext^{\Bip}_{f,g} \co \Id \big( (f, \bar{f}_0, \bar{f}_1), (g, \bar{g}_0, \bar{g}_1) \big) \to 
\BipHot\big( (f, \bar{f}_0, \bar{f}_1), (g, \bar{g}_0, \bar{g}_1) ) \big) \, .
\]
is an equivalence of types.
\end{lemma}

\begin{proof} Recall that, for a path $p \co \Id( f, g)$, we write $\ext \, p \co \Hot(f, g)$
for the corresponding homotopy. We then have
\begin{align*}
 \Id \big( (f,\bar{f}_0,\bar{f}_1) , (g,\bar{g}_0,\bar{g}_1)  \big)
  &  \simeq (\Sigma p \co \Id( f, g))  \, 
  \Id \big(  (\bar{f}_0, p^*(\bar{g}_0) \big) \times \Id \big( \bar{f}_1,   p^{\ast} (\bar{g}_1) \big) \\
 & \simeq (\Sigma p  \co \Id(f,g))  \, \Id(\bar{f}_0,  \bar{g}_0 \ct (\ext \, p)_{a_0} ) \times \Id( \bar{f}_1, 
 \bar{g}_1\ct (\ext \, p)_{a_1} ) \\
& \simeq (\Sigma \alpha \co \Hot(f, g)) \,  \Id(\bar{f}_0,  \bar{g}_0 \ct \alpha_{a_0} ) \times \Id(\bar{f}_1, \bar{g}_1
\ct \alpha_{a_1} ) \\
& = \BipHot \big( (f,\bar{f}_0,\bar{f}_1) \; (g,\bar{g}_0,\bar{g}_1) \big) \, , 
\end{align*} 
as required.
\end{proof}

\subsection*{Fibered bipointed types and bipointed sections} Recall that for a dependent type
\[
 x \co A \vdash E(x) \co \type
\]
we referred to an element $f \co (\Pi x \co A) E(x)$ as a \emph{section} of the dependent type.
It will be convenient to extend this notion  to bipointed types  by introducing the following definition. 

\medskip

Let us fix a bipointed type $A = (A, a_0, a_1)$. 

\begin{definition} \label{def:fibbipointed}
A \emph{fibered bipointed type} $(E, e_0, e_1)$ over $A$  is a dependent type~$x\co A \vdash E(x) \co \type$ 
equipped  with elements $e_0 \co E(a_0)$ and $e_1 \co E(a_1)$.
\end{definition}

The type of  small fibered bipointed types over a bipointed type $A$ is then defined by letting
\[
\FibBip(A) \defeq (\Sigma E \co A \to \U) \big( E(a_0) \times E(a_1) \big) \, .
 \]
 
 \medskip

Let us now fix a fibered bipointed type $E = (E, e_0, e_1)$  over $A$.
 
 \medskip
 
The type
$E'  \defeq (\Sigma x \co A) E(x)$ can be equipped with the structure of
a a bipointed type by considering $e'_k \defeq \pair(a_k, e_k)$ 
(for $k \in \{ 0, 1 \}$) as distinguished elements of $E'$. In this way, the first projection $\pi_1 \co E' \to A$ becomes a bipointed morphism:
\[
\xymatrix@C=2cm{
1  \ar[r]^-{e'_0} \ar@{=}[d] \ar@{}[dr]|{\overline{(\pi_1)}_0} & E' \ar[d]^{\pi_1} & 1 \ar[l]_-{e'_1} \ar@{=}[d]
\ar@{}[dl]|{\overline{(\pi_1)}_1}   \\ 
 1 \ar[r]_{a_0} & A  & \; 1 \, . \ar[l]^{a_1} }
 \]

\begin{definition} \label{def:fibsection} A \emph{bipointed section} $(f, \bar{f}_0, \bar{f}_1)$ of $E$
is a section $f \co (\Pi x \co A) E(x)$ equipped with paths $\bar{f}_0 \co \Id_{E(a_0)}(f a_0, e_0)$ 
and $\bar{f}_1 \co  \Id_{E(a_1)}( f a_1 , e_1)$. 
\end{definition}

The type of bipointed sections of $E$ is then defined by letting
\[
\BipSec(A,E) \defeq (\Sigma f \co (\Pi x \co A)E(x) ) \; \big(
  \Id_{E(a_0)}(f a_0,  e_0)  \times \Id_{E(a_1)}( fa_1, e_1)  \big) \, .
\]
Given a bipointed section $f = (f, \bar{f}_0, \bar{f}_1)$ of $E$, we can define a bipointed morphism~$f'  \co A \to E'$, where $E' =
(E', e'_0, e'_1)$ is  the bipointed type associated to $E$. Its underlying function is defined
by $f' \defeq (\lambda x \co A) \pair(x, fx)$. With this definition, it is 
immediate to get the required paths~$\bar{f'}_k \co \Id( f' a_k ,  e'_k)$, for $k \in \{ 0, 1\}$. Note that
the morphism $f' \co A \to E'$ provides a  right inverse for~$\pi_1 \co E' \to A$,
 since for every $x \co A$ we have the judgemental equalities $\pi_1 (f' x) \deq \pi_1\, \pair(x, f x) \deq x$.
 We represent this situation with the diagram
\[
\xymatrix{
E' \ar[d]_{\pi_1} \\
A . \ar@/_1.2pc/[u]_{f'} }
\]
We characterize the identity type between two bipointed sections, using
the notion of a bipointed homotopy. This is in complete analogy with what was done
for bipointed morphisms  in Lemma~\ref{BoolHomSpace}. 

\medskip

Let us now fix two bipointed sections $f = (f, \bar{f}_0,\bar{f}_1)$ and $g = (g, \bar{g}_0, \bar{g}_1)$  of $E$.

\begin{definition} \label{def:2cellsection} A \emph{bipointed homotopy}  $(\alpha, \bar{\alpha_0},
\bar{\alpha}_1) \co f \to g$ is  a homotopy~$\alpha \co \Hot(f, g)$ equipped with paths $\bar{\alpha}_0 \co \Id(  \bar{f}_0 ,  \bar{g}_0 \ct \alpha_{a_0}  )$ and $\bar{\alpha}_1 \co \Id(
\bar{f}_1 , \bar{g}_1 \ct  \alpha_{a_1})$.
\end{definition}

The type of bipointed homotopies between $f$ and $g$ as above is then defined by letting:
\[
\BipHot \big( (f, \bar{f}_0, \bar{f}_1), (g, \bar{g}_0, \bar{g}_1) \big) \defeq
(\Sigma \alpha \co \Hot( f, g)) \, \big( 
\Id\big( \bar{f}_0 ,   \bar{g}_0 \ct \alpha_{a_0}  \big) \times 
  \Id \big( \bar{f}_1,    \bar{g}_1 \ct \alpha_{a_1}  \big) \big)   \, .
\]

\begin{lemma} \label{thm:biphot}
The canonical function
\[
\ext^{\BipHot}_{f,g} \co \Id\big( (f, \bar{f}_0, \bar{f}_1), (g, \bar{g}_0, \bar{g}_1) \big) \rightarrow
\BipHot \big( (f, \bar{f}_0, \bar{f}_1), (g, \bar{g}_0, \bar{g}_1) \big) 
\]
is an equivalence of types. 
\end{lemma}

\begin{proof} The claim follows by an argument analogous to that of Lemma~\ref{BoolHomSpace}.
\end{proof}

\subsection*{Bipointed equivalences} We introduce the notion of equivalence between bipointed types and show
in Proposition~\ref{thm:usemere}  that a bipointed morphism is an equivalence of bipointed types if and only its underlying function is an equivalence of types. For this, we will use the characterization of equivalence of types
as functions with a left and right inverse, which we recalled in Section~\ref{sec:bac}. The characterization of
bipointed equivalences given below will be used in Section~\ref{sec:unibip} where we consider the counterpart of the univalence axiom for bipointed types.

\begin{definition} We say that a bipointed morphism $f \co A \to B$ is a \myemph{bipointed equivalence}
if there exist bipointed morphisms $g \co B \to A$ and $h \co B \to A$ which provide a left and a right bipointed inverse for $f$, \ie such that there exist paths $p \co \Id_{\Bip(A,A)}(g  f, 1_A)$ 
and $q \co \Id_{\Bip(B,B)} ( f  h, 1_B)$.
\end{definition}

For a bipointed morphism $f \co A \to B$, the type of proofs that $f$ is a bipointed equivalence is
then defined by letting
\[
\isbipequiv(f) \defeq   (\Sigma g \co \BipHom(B,A)) \,  \Id_{\Bip(A,A)}( g  f, 1_A ) \times 
    (\Sigma h \co \BipHom(A, B)) \, \Id_{\Bip(B,B)} (f  h , 1_B ) \, ,
\]
and type of bipointed equivalences between $A$ and $B$ is defined by letting
\[
\BipEquiv(A, B)
\defeq    
(\Sigma f \co \BipHom(A,B)) \, \isbipequiv(f)  \, . 
\]

\begin{lemma} The underlying function of a bipointed equivalence is an equivalence of types. In particular,
for every bipointed morphism $f \co A \to B$  there is a  function 
\[
\pi_f \co \isbipequiv(f) \to \isequiv(f) \, .
\]
\end{lemma} 

\begin{proof} Let $f = (f, \bar{f}_0, \bar{f}_1)$ be a bipointed morphism from $A$ to $B$. Unfolding  the definition
of  $\isbipequiv(f)$ yields the type
\begin{multline}
\label{equ:unfoldisbipequiv} 
(\Sigma g \co B \to  A)
(\Sigma \bar{g}_0 \co \Id( g b_0, a_0)) 
(\Sigma  \bar{g}_1 \co \Id( g b_1, a_1)) \, 
 G(g,\bar{g}_0,\bar{g}_1) \, \times \\
 (\Sigma h \co B \to A)
 (\Sigma \bar{h}_0 \co \Id ( h b_0,  a_0))
 (\Sigma \bar{h}_1 \co  \Id ( h b_1,  a_1)) \, 
 H(h,\bar{h}_0,\bar{h}_1)   \, ,
\end{multline}
where 
\begin{align*}
G(g,\bar{g}_0,\bar{g}_1) & \defeq 
\Id 
\big( \, 
( 
g f \, ,  \overline{gf}_0  \, ,    \overline{gf}_1 
) \, , \;
( 
1_A \, ,  1_{a_0} \, , 1_{a_1} 
) \, 
\big)  \, , \\
H(h,\bar{h}_0,\bar{h}_1) & \defeq 
\Id 
\big( 
(
 f  h \, ,   \overline{fh}_0 \, , \overline{fh}_1   
 ) \, , \; 
 ( 
 1_B \, ,  1_{b_0} \, , 1_{b_1} 
 ) 
 \big) \, .
\end{align*}
The type $G(g, \bar{g}_0, \bar{g}_1)$ can be thought of as the type of proofs that the
bipointed morphism  $g  f \co A \to A$ is propositionally equal to the identity bipointed morphism
$1_A \co A \to A$, while  $H(h, \bar{h}_0, \bar{h}_1)$ can be thought of as the type of proofs that the
bipointed morphism  $f  h \co B \to B$ is propositionally equal to the identity bipointed morphism
$1_B \co B \to B$. In particular, the elements of $G(g, \bar{g}_0, \bar{g}_1)$ can be thought of as
proofs that the pasting diagram 
\[
\xymatrix@C=2cm@R=1.2cm{
1    \ar[r]^{a_0} \ar@{=}[d]  \ar@{}[dr]|{\Downarrow \, \bar{f}_0} & A  \ar[d]^{f} & 1  \ar[l]_{a_1} \ar@{=}[d] \ar@{}[dl]|{\Downarrow \,  \bar{f}_1} \\
1 \ar[r]_{b_0}   \ar@{}[dr]|{\Downarrow \, \bar{g}_0}  \ar@{=}[d] & B \ar[d]^g   & 1 \ar[l]^{b_1} \ar@{=}[d] \ar@{}[dl]|{\Downarrow \,  \bar{g}_1} \\
1 \ar[r]_{a_0}  & A   & 1 \ar[l]^{a_1}}
\]
is propositionally equal to the diagram in~\eqref{equ:bipidA} representing the identity bipointed morphism. 

Using the characterization of identity types of $\Sigma$-types in Section~\ref{sec:chait}, the type $G(g,\bar{g}_0,\bar{g}_1)$ can be  expressed equivalently as
\[
(\Sigma p \co  \Id( g f, 1_A)) \, 
\Id 
\big(  
( 
\overline{gf}_0  \, ,    \overline{gf}_1 
) \, ,  \; 
 p^*
 ( 
 1_{a_0} \, , 1_{a_1}  
 )  
 \big) \, ,
\]
where, for $p \co \Id(gf, 1_A)$, 
\[
p^* \co \big(  \Id( 1_A(a_0), a_0) \times \Id(1_A(a_1), a_1) \big) \to  \big( \Id( gf(a_0), a_0) \times \Id( gf(a_1), a_1)  \big)
\]
is a transport function associated to $p$. Similarly, the type~$H(h, \bar{h}_0, \bar{h}_1)$ is equivalent to 
\[
(\Sigma q \co  \Id( f h, 1_B)) \, 
\Id 
\big(  
( 
\overline{fh}_0  \, ,    \overline{fh}_1 
) \, ,  \; 
 q^*
 ( 
 1_{b_0} \, , 1_{b_1}  
 )  
 \big) \, .
\]
Thus, rearranging the order of the $\Sigma$-types in~\eqref{equ:unfoldisbipequiv} and using the
characterization of identity types in product types, we get that 
\begin{multline}
\label{equ:equivdescisbipequiv}
\isbipequiv(f) \simeq  \\ (\Sigma g \co B \to A)(\Sigma p \co \Id(gf, 1_A))\,  G'(g,p) \times
(\Sigma h \co B \to A)(\Sigma q \co \Id(fh, 1_B))\, H'(h,q) \, , 
\end{multline}
where 
\begin{align} 
G'(g,p) & \defeq 
(\Sigma \bar{g}_0 \co \Id( g b_0, a_0)) \, \Id( \overline{gf}_0, p^*(1_{a_0}) ) \times 
(\Sigma  \bar{g}_1 \co \Id( g b_1, a_1)) \, \Id( \overline{gf}_1, p^*(1_{a_1}) ) \, ,
 \label{equ:bipgp} \\
 H'(h,q) & \defeq 
 (\Sigma \bar{h}_0 \co \Id ( h b_0,  a_0)) \, \Id ( \overline{fh}_0 , q^*(1_{b_0})) \times 
 (\Sigma \bar{h}_1 \co  \Id ( h b_1,  a_1)) \, \Id (   \overline{fh}_1 \, ,  q^*(1_{b_1})) \, .  \label{equ:biphq}
 \end{align} 
 Note that the elements of~$G'(g,p)$ are 4-tuples consisting of paths $\bar{g}_0$, $\bar{g_1}$ making the function~$g$ into a bipointed morphism and of paths $\bar{p}_0$, $\bar{p}_1$ making the path $p$ into a path between
bipointed morphisms. Of course, the elements $H'(h,q)$ admits a similar description. 
The required  function  is then obtained 
by composing the equivalence in~\eqref{equ:equivdescisbipequiv}, the projection 
forgetting the components from $G'(g,p)$ and $H'(h,q)$, and the equivalence in~\eqref{equ:equivalternative}.
\end{proof}

In Proposition~\ref{thm:usemere} we will give an alternative characterisation of bipointed 
equivalences, which will be used in the proof of Theorem~\ref{thm:bipunivalence} and Corollary~\ref{BoolHInitIso}.
Intuitively, it asserts that for every bipointed morphism $(f, \bar{f}_0, \bar{f}_1)$, 
 if the underlying function $f$  is an equivalence of types, there is an 
essentially unique way of making $(f, \bar{f}_0, \bar{f}_1)$ into a bipointed equivalence, \ie of equipping the left and right inverses of $f$ with the structure of bipointed morphisms so as to obtain bipointed inverses\footnote{This has several analogues in category theory. For example, consider monoidal categories $\mathbb{C}$ and $\mathbb{D}$
and a strong monoidal functor $F \co \mathbb{C} \to \mathbb{D}$ which is an equivalence of categories. There is then an
essentially unique way of making a quasi-inverse of $F$ into a strong monoidal functor so as to obtain a
monoidal equivalence.}  In order to prove this result,
we need the following straightforward lemma.

\begin{lemma} \hfill \label{thm:useful}
\begin{enumerate}[(i)]
\item Let $A$ be a type and $a, a_1, a_2 \co A$. For paths $p_1 \co \Id(a,a_1)$, $p_2 \co \Id(a,a_2)$, the type 
\[
(\Sigma q \co \Id_A(a_1,a_2)) \, \Id( q \ct p_1 \, , p_2)
\] 
is contractible. 
\item Let $f \co A \to B$ be an equivalence, $a_1, a_2 \co A$ and $b \co B$. For paths $p_1 \co \Id(b, f a_1)$, 
$p_2 \co \Id(b, fa_2)$, the type 
\[
(\Sigma q \co \Id_A(a_1,a_2)) \, \Id( (f \circ q) \ct p_1 , p_2)
\] 
is contractible.
\end{enumerate}
\end{lemma}

\begin{proposition}  \label{thm:usemere}  A bipointed morphism $(f, \bar{f}_0, \bar{f}_1) \co A \to B$ is a bipointed equivalence if and only
if its underlying function $f \co A \to B$ is an equivalence. In fact,  the  function
\[
\pi_f \co \isbipequiv(f, \bar{f}_0, \bar{f}_1)  \to \isequiv(f) \, .
\]
is an equivalence of types. 
\end{proposition}  

\begin{proof}
Let  $(f, \bar{f}_0, \bar{f}_1) \co A \to B$ be a bipointed morphism. We wish to show that the homotopy fibers
of the function $\pi_f$ are contractible. So, let us fix a canonical element of $\isequiv(f)$, given by functions $g \co B
\to A$, $h \co B \to A$ and paths $p \co \Id(gf, 1_A)$ and $q \co \Id(fh, 1_B)$. By the definition of~$\pi_f$ and standard
facts about the homotopy fibers, we have an equivalence 
\[
\hfiber(\pi_f, (g,h, p, q)) \simeq G'(g,p) \times H'(h,q) \, ,
\]
 where $G'(g,p)$ and $H'(h,q)$ are defined
in~\eqref{equ:bipgp} and~\eqref{equ:biphq}, respectively. 
We claim that $G'(g,p)$ and $H'(h,q)$ are contractible. Since
the proofs are essentially the same, we consider only $G'(g,p)$.  

Let $k \in \{0, 1 \}$. For a path $p \co \Id(gf, 1_A)$, the path $p^*(1_{a_k}) \co \Id( gf a_k, b_k)$ can be proved by $\Id$-elimination to be propositionally equal to $(\ext \, p)_{a_k}\co \Id(gfa_k, b_k)$, where $\ext \, p \co \Hot(gf, 1_A)$. Combining this fact with the definition of composition of bipointed morphisms, we obtain that
$G'(g,p)$ is equivalent to the product of the types
\[
(\Sigma \bar{g}_k \co \Id( gb_k, a_k)) \, \Id ( \bar{g}_k \ct (g \circ \bar{f}_k), (\ext \, p)_{a_k}) \, ,
\]
for $k \in \{0, 1 \}$, which are contractible by part~(i) of Lemma~\ref{thm:useful}. Hence $G'(h,p)$ is contractible, as required.
\end{proof}

\begin{corollary} For any bipointed morphism $(f, \bar{f}_0, \bar{f}_1)$, the type $\isbipequiv(f, \bar{f}_0, \bar{f}_1)$ is a mere proposition. \qed
\end{corollary}

\section{Homotopy-initial bipointed types} 
\label{sec:homibt}

\subsection*{Inductive bipointed types} 
As we mentioned at the beginning of Section~\ref{sec:biptm}, if a type $A$ is equivalent to~$\Bool$, then 
it satisfies the counterparts of the elimination and computation rules for $\Bool$ in which the computation rule is 
weakened by replacing the judgmental equality in its conclusion with a propositional equality. Using the notions of a fibered bipointed type and of a bipointed section introduced in Section~\ref{sec:bip}, it is immediate to see that the these rules rules can be expressed equivalently by saying  that every fibered  bipointed type over $A$ has a bipointed 
section (cf.~\cite{Joyal:cathtt}). Since bipointed types $A$ of this kind  play an important role in the following, we introduce some terminology\footnote{We use `inductive' in analogy with the terminology used in set theory. This is not to be confused with the general notion of an inductive type.} to refer to them.

\begin{definition} A bipointed type $A$ is said to be \emph{inductive} if every small fibered bipointed type over it has a bipointed section, \ie the type
\[ 
\isbipind(A) \defeq (\Pi E \co \FibBip(A))  \BipSec(A,E)
\]  
is inhabited. \end{definition} 

As we will see in Proposition~\ref{thm:isbipindisprop}, the type $\isbipind(A)$ is a mere proposition.
We define the type of small inductive bipointed types by letting
\[
\mathsf{BipInd} \defeq (\Sigma A \co \Bip) \isbipind(A) \, .
\]
Thus, a canonical inductive bipointed type is given by a bipointed type $A = (A, a_0, a_1)$ together with a function 
which, given a fibered bipointed type $E = (E, e_0, e_1)$ over $A$, returns a bipointed section of $E$.
Clearly, the type $\Bool$ is an inductive bipointed type. Furthermore, the property of being inductive can be transported along equivalences, in the sense that if $A$ and $B$ are equivalent bipointed types and $A$ is inductive, then so is $B$. Thus, a
type is equivalent to $\Bool$ if and only if it is inductive. Below,
we begin to explore some consequences of the assumption that a bipointed type is inductive, with the goal
of arriving at a 
characterisation of inductive bipointed types in Theorem~\ref{thm:bipointedmain}.

\begin{proposition} \label{thm:inductiverules}
Let $A = (A, a_0, a_1)$ be a bipointed type. Then $A$ is inductive if only if it satisfies the following rules:
\begin{enumerate}[(i)]
\item the elimination rule
\[
\begin{prooftree}
x \co A \vdash E(x) \co \U \qquad
e_0 \co E(a_0) \qquad
e_1 \co E(a_1) 
\justifies
x \co A \vdash \elim(x, e_0, e_1) \co E(x) \, , 
\end{prooftree} 
\]
\item the computation rules 
\[
\begin{prooftree}
x\co A \vdash E(x) \co \U \qquad
e_0 \co E(a_0) \qquad
e_1 \co E(a_1)
\justifies
\comp_k(e_0, e_1) \co \Id \big(    \elim(a_k, e_0, e_1), e_k \big) \, ,
\end{prooftree}  
\]
where $k \in \{ 0, 1\}$.
\end{enumerate}
\end{proposition}

\begin{proof} Immediate.
\end{proof}

In the following, when we speak of an inductive bipointed type, we always assume that it
comes equipped with functions $\elim$ and $\comp_k$ (for $k \in \{0, 1\}$) as in Proposition~\ref{thm:inductiverules}.
Note that the rules in Proposition~\ref{thm:inductiverules} are exactly the counterparts for $A$ of 
the elimination rule and the weakening computation rules for $\Bool$ obtained by restricting the eliminating
type to families of small dependent types\footnote{See Remark~\ref{thm:smallelimination} for further discussion of
this point.} and, most importantly,  replacing the judgemental
equality in the conclusion with a propositional one, as mentioned above. The next proposition
shows that, for an inductive bipointed type~$A$,  not only every fibered bipointed type over it
has a section, but that such a section is unique up to a bipointed homotopy.

\begin{proposition} \label{thm:inductiveuniquesec} Let $A = (A, a_0, a_1)$ be a bipointed type. If $A$ is inductive, 
then the following rules are derivable:

\begin{enumerate}[(i)]
\item the $\eta$-rule
\[
\begin{prooftree}
\begin{array}{c} 
x \co A \vdash E(x) \co \U \quad
e_0 \co E(a_0) \quad
e_1 \co E(a_1) \quad
x \co A \vdash f x \co E(x) \quad
\bar{f}_0 \co \Id(fa_0, e_0) \quad
\bar{f}_1 \co \Id(fa_1, e_1)
\end{array}
\justifies
x \co A \vdash \eta_x \co \Id(f x,  \elim(x, e_0, e_1))
\end{prooftree} \bigskip
\]
\item the coherence rule
\[
\begin{prooftree}
\begin{array}{c} 
x \co A \vdash E(x) \co \U \quad
e_0 \co E(a_0) \quad
e_1 \co E(a_1) \quad
x \co A \vdash f(x) \co E(x) \quad
\bar{f}_0 \co \Id(f a_0, e_0) \quad
\bar{f}_1 \co \Id(f a_1, e_1)
\end{array}
\justifies
\bar{\eta}_k \co \Id \big( \comp_k(e_0, e_1)  \ct \eta_{a_k} \, , \;  \bar{f}_k \big)
\end{prooftree}
\]
\noindent
where $k \in \{ 0, 1 \}$.
\end{enumerate}
\end{proposition}

Before proving Proposition~\ref{thm:inductiveuniquesec}, observe that the paths in the conclusion of 
the coherence rule can be represented diagrammatically in a way that is reminiscent of one of the triangular
laws for an adjunction\footnote{See Remark~\ref{thm:etarule} for further discussion of this analogy.}:
\[
\xymatrix@C=1.5cm{
f a_k  \ar[r]^-{\eta_{a_k}}  \ar@/_1pc/[dr]_{\bar{f}_k}  
\ar@{}[dr]|{\qquad \Downarrow \; \bar{\eta}_k}  & \elim(a_k, e_0, e_1) \ar[d]^{\varepsilon_k}  \\ 
 & \; e_k  \, , }
  \] 
  where $\varepsilon_k \defeq \comp_k(e_0, e_1)$, for $k \in \{ 0, 1 \}$.

\begin{proof}[Proof of Proposition~\ref{thm:inductiveuniquesec}] Let us assume the premisses of the $\eta$-rule. For $x \co A$, define $F(x) \co \U$  by letting $F(x) \defeq 
\Id_{E(x)}(fx, \elim(x, e_0, e_1))$. 
With this notation, proving the conclusion of the $\eta$-rule amounts to defining
$\eta_x \co F(x)$, for $x \co A$. We do so using the elimination rule for $A$, as stated in Proposition~\ref{thm:inductiverules}.
Thus, we need to find elements $p_k \co F(a_k)$, for $k \in \{0, 1\}$. Since
\[
F(a_k) = \Id(fa_k, \elim(a_k, e_0, e_1)) \, ,
\]
we define $p_k$ as the composite
\[
\xymatrix@C=2.2cm{
 fa_k \ar[r]^{\bar{f}_k} &
 e_k \ar[r]^-{ \comp_k(e_0, e_1)^{-1}}& 
  \elim(a_k, e_0, e_1)  \, .}
\]
For $x \co A$, we can then defined the required element $\eta_x \co F(x)$ by letting $\eta_x \defeq \elim(x, p_0, p_1)$.
In order to prove the coherence rule, note that the computation rule of Proposition~\ref{thm:inductiverules} gives us a path in $\Id(\eta_{a_k},  p_k)$, \ie  $\Id( \eta_{a_k},  \comp_k(e_0, e_1)^{-1} \ct \bar{f}_k ) $. 
The required paths can then be obtained using the groupoid laws.
 \end{proof} 

\begin{proposition} \label{thm:isbipindisprop} For every bipointed type $A = (A, a_0, a_1)$, the type $\isbipind(A)$ is a mere proposition.
\end{proposition}

\begin{proof} Recall that to prove that a type  is a mere proposition, it suffices to do so under the assumption that it is inhabited. Assume therefore that~$\isbipind(A)$ is inhabited. Since the dependent product of a family of mere propositions is again a mere proposition, it suffices to show that~$\BipSec(A,E)$ is a mere proposition for any $E$. But for any two bipointed sections~$f, g \co \BipSec(A,E)$,  there is a 
bipointed homotopy $\alpha \co \BipHot(f,g)$ by Proposition~\ref{thm:inductiveuniquesec} and hence, by 
Lemma~\ref{thm:biphot}, there is a path $p \co \Id(f,g)$, as required. 
\end{proof}

\subsection*{Homotopy-initial bipointed types} 
 Let $A$ be a small bipointed type and assume that it is inductive. 
 We focus on the special case of fibered bipointed types that 
are constant, \ie we have $E(x) = B$ for all $x \co A$, where $B = (B, b_0, b_1)$  is
a small bipointed type. 
 Proposition~\ref{thm:inductiverules} and Proposition~\ref{thm:inductiveuniquesec}
imply that there exists a bipointed morphism~$f \co A \to B$, which is unique in the sense that  for any two bipointed morphisms $(f, \bar{f}_0, \bar{f}_1), (g, \bar{g}_0, \bar{g}_1) \co A \to B$  there is a bipointed 
homotopy~$\alpha \co \Hot(f, g)$. Thus, by Lemma~\ref{BoolHomSpace}, there is a path 
\[
p \co \Id((f, \bar{f}_0, \bar{f}_1), (g, \bar{g}_0, \bar{g}_1)) \, .
\] 
Furthermore, it can be shown that such a path is itself unique up to a higher path, which in turn is unique up to a yet higher path, and so on. 

The key point in our development (described for $\Bool$ below and for $\W$-types in Section~\ref{sec:homta}) is that this sort of weak $\infty$-universality, which apparently involves infinitely much data, can be captured \emph{fully within the system of type theory} (without resorting to coinduction) using ideas inspired by homotopy theory and higher-dimensional category theory. Indeed, in spite of the fact that bipointed types and morphisms do not form a category in a strict sense, it is possible to introduce the  notion of a homotopy-initial bipointed type in completely elementary and explicit terms, as in Definition~\ref{def:BoolInit} below. This provides the template for the definition of a homotopy-initial algebra, which we will introduce in Section~\ref{sec:homta} in relation to $\W$-types.

\begin{definition}\label{def:BoolInit}
A small bipointed type $A$ is said to be \emph{homotopy-initial}  if for any small bipointed type $B$, the type $\BipHom(A,B)$ of bipointed morphisms from $A$ to $B$
is contractible, \ie the type
\[
\ishinit(A) \defeq (\Pi B \co \Bip) \, \iscontr(\BipHom(A, B) )
\] 
is inhabited.
\end{definition}

Let us remark that the uniqueness implicit in Definition~\ref{def:BoolInit} requires that any two bipointed morphisms are propositionally equal as tuples. It should also be noted that the property of being  homotopy-initial  can be transported along equivalences, in the sense that if two bipointed types are equivalent, then one is homotopy-initial if and only if the other one is.  

\begin{proposition} \label{thm:isbiphinitishprop} For every bipointed type $A$, the type $\ishinit(A)$ is a mere proposition.
\end{proposition}

\begin{proof} Recall that, for a type $X$, the type $\iscontr(X)$ is a mere proposition and that the dependent product of family of mere propositions is again a mere proposition. 
\end{proof} 

The next result is the counterpart of the familiar fact that objects characterized by universal properties are unique up to a unique isomorphism.

\begin{proposition} \label{BoolHInitIso} 
Homotopy-initial small bipointed types are unique up to a contractible type of bipointed equivalences, i.e. the type
\[ 
(\Pi A \co \Bip) (\Pi B \co \Bip)
\big( \isbiphinit(A) \times \isbiphinit(B) \to \iscontr(\BipEquiv(A,B)) \big) \, .
\] 
is inhabited.
\end{proposition}

\begin{proof} 
Let $A$ and $B$ be homotopy-initial bipointed types. The type 
$\BipHom( A, B)$ is contractible by homotopy-initiality of $A$. Since the dependent sum of a family of mere propositions over a mere proposition is again a mere proposition, it suffices to prove that the type $\isbipequiv(f)$ is contractible for any bipointed morphism $f \co A \to B$. This type is a mere proposition by 
Proposition~\ref{thm:usemere}, and thus it suffices to show it is inhabited. But the existence of a right and a left bipointed inverse for $f$ follows immediately
by the assumption that $A$ and $B$ are homotopy-initial.
\end{proof}

The next proposition spells out a  characterization of homotopy-initial bipointed types in terms of type-theoretic rules.

\begin{proposition} \label{thm:hinitrules}
A small bipointed type $A = (A, a_0, a_1)$ is homotopy-initial if and only if it satisfies
 the following rules:
 
 \begin{enumerate}[(i)]
 \item the recursion rule
 \[
\begin{prooftree}
B \co \U \qquad
b_0 \co B \qquad
b_1 \co B 
\justifies
x \co A \vdash \rec(x, b_0, b_1) \co B \, , 
\end{prooftree} 
\]
\item the $\beta$-rules
\[
\begin{prooftree}
B \co \U \qquad
b_0 \co B  \qquad
b_1 \co B
\justifies
\beta_k \co \Id(  \rec(a_k, b_0, b_1), b_k ) \, , 
\end{prooftree}  
\]
where $k \in \{0, 1\}$, 
\item the $\eta$-rule
\[
\begin{prooftree}
(B, b_0, b_1) \co \Bip \quad
(f, \bar{f}_0, \bar{f}_1) \co \Bip(A,B)
\justifies
x \co A \vdash \eta_x \co \Id(fx, \rec(x, b_0, b_1) ) \, , 
\end{prooftree}  
\]
\item the $(\beta, \eta)$-coherence rule
\[
\begin{prooftree}
(B, b_0, b_1) \co \Bip \quad
(f, \bar{f}_0, \bar{f}_1) \co \Bip(A,B) 
\justifies
\bar{\eta}_k \co \Id( \beta_k \ct \eta_{a_k} \, , \; \bar{f}_k) \, , 
\end{prooftree}
\]
 where $k \in \{ 0, 1 \}$.
 \end{enumerate}
\end{proposition}

\begin{proof} The claim follows by unfolding the definition of homotopy-initiality.
\end{proof} 

\begin{remark} \label{thm:etarule} The terminology used for the rules in Proposition~\ref{thm:hinitrules} is
inspired by the special case that arises by considering $B$ to be $A$ itself and $f$ to be the identity function.
In this case, we obtain a function $(\lambda x \co A) \rec(x, a_0, a_1) \co A \to A$, paths $\beta_k \co
\Id( \rec(a_k, a_0, a_1), a_k)$, for $k \in \{ 0, 1 \}$ and $\eta_x \co \Id(x, \rec(x, a_0, a_1)$ and higher paths
$\bar{\eta}_k$ fitting in the diagram
\[
\xymatrix@C=1.5cm{
a_k \ar[r]^-{\eta_{a_k}} \ar@/_1pc/[dr]_{1_{a_k}} \ar@{}[dr]|{\qquad \overset{\bar{\eta}_k\; }{\Rightarrow}} & \rec(a_k, a_0, a_1) \ar[d]^{\beta_k} \\ 
 & a_k\, ,}
 \]
 for $k \in \{ 0, 1\}$,  which are analogous to one of the triangle laws for an adjunction. 
\end{remark}

The next theorem provides a characterisation of inductive bipointed types.

\begin{theorem}\label{thm:bipointedmain} A small bipointed type 
 $A$ is inductive if and only if it is homotopy-initial, \ie  the type
\[
(\Pi A \co \Bip) \big(  \isbiphinit(A) \leftrightarrow \isbipind(A) \big)
\] 
is inhabited
\end{theorem}

\begin{proof} Let $A = (A, a_0, a_1)$ be a small bipointed type. We prove the two 
implications separately. 

First, we show that if $A$ is inductive then
it is homotopy-initial. For this, it is sufficient to 
observe that the rules characterizing homotopy-initial bipointed types in Proposition~\ref{thm:hinitrules}
are special cases of the rules in Proposition~\ref{thm:inductiverules} and Lemma~\ref{thm:inductiveuniquesec},
which are provable for inductive bipointed types. 

Secondly, let us assume that $A = (A, a_0, a_1)$ is homotopy-initial and prove that it is inductive. 
For this, let~$E = (E, e_0, e_1)$ be a fibered small bipointed type over $A$. We need to show that there
exists a bipointed section $(s, \bar{s}_0, \bar{s}_1) \co \BipSec(A,E)$.
Let us consider the bipointed type associated to $E$, with carrier 
$E' \defeq (\Sigma x \co A) E(x)$ 
and distinguished elements  $e'_k \defeq \pair(a_k, e_k)$, 
for~$k \in \{ 0, 1 \}$. In this way,  the first projection $\pi_1 \co 
E' \to A$ is a bipointed morphism. By the homotopy-initiality of $A$, we have a bipointed morphism 
$(f, \bar{f}_0, \bar{f}_1) \co (A, a_0, a_1)  \to (E', e'_0, e'_1)$, 
which we represent with the diagram
\[
\xymatrix@C=2cm{
1 \ar[r]^{a_0} \ar@{=}[d] \ar@{}[dr]|{\Downarrow \, \bar{f}_0} & A  \ar[d]^f & 1 \ar[l]_{a_1} \ar@{=}[d]
 \ar@{}[dl]|{\Downarrow \, \bar{f}_1}   \\
1 \ar[r]_{e'_0}  & E' & 1 \ar[l]^{e'_1} }
 \]
 We can compose $f \co A \to E'$ with $\pi_1 \co E' \to A$ and obtain a bipointed morphism $\pi_1  f \co A \to A$, which is represented by the diagram
  \[
\xymatrix@C=2cm@R=1.2cm{
1  \ar[r]^{a_0} \ar@{=}[d]  \ar@{}[dr]|{\Downarrow \, \bar{f}_0}  & A  \ar[d]^f & 1 \ar[l]_{a_1}  \ar@{=}[d]  
 \ar@{}[dl]|{\Downarrow \, \bar{f}_1} \\
\ar[r]^{e'_0} \ar@{=}[d]   \ar@{}[dr]|{\Downarrow \, \bar{\pi}_0}  & E' \ar[d]^{\pi_1}  & 1 \ar[l]_{e'_1} \ar@{=}[d] 
 \ar@{}[dl]|{\Downarrow \, \bar{\pi}_1}  \\
1 \ar[r]_{a_0} & A &  \; 1  \, . \ar[l]^{a_1}}
 \]
Since the identity $1_A \co A \to A$ is also a bipointed morphism, by the homotopy-initiality of $A$ there is an element of $\Id_{\Bip(A,A)}(\pi_1  f, 1_A)$. 
 By Lemma~\ref{BoolHomSpace}, this gives us a bipointed homotopy $(\alpha,
\bar{\alpha}_0,\bar{\alpha}_1) \co \BipHot( \pi_1  f , 1_A)$. This amounts to a homotopy $\alpha \co \Hot(\pi_1 
 f, 1_A)$ and paths
\begin{equation*}
\bar{\alpha}_k \co  \Id( \overline{(\pi_1  f)}_k \, , \;  \alpha_{a_k} \cdot 1_{a_k} ) \, ,
\end{equation*}
for $k \in \{ 0, 1 \}$. We begin to define the required bipointed section by defining, for~$x \co A$, 
\begin{equation*}
s(x) \defeq (\alpha_x)_{!} \big( \pi_2 f  x \big) \, ,
\end{equation*}
where $(\alpha_x)_{!} \co E(\pi_1 f x) \to E(x)$. 
We now construct  paths~$\bar{s}_k \co \Id(s a_k, e_k)$, for $k \in \{ 0, 1 \}$. First of all, recall that 
$\bar{f}_k \co \Id( f a_k, e'_k)$, where $e'_k = \pair(a_k, e_k) \co (\Sigma x \co A) E(x)$. Using the 
characterization of identity types of $\Sigma$-types, we define
\[
p \defeq \pi_1 \, \ext^\Sigma(\bar{f}_k)  \co  \Id( \pi_1 f  a_k \, ,\;   \pi_1 e'_k )  \, ,  \quad 
q \defeq  \pi_2 \, \ext^\Sigma(\bar{f}_k) \co   \Id( p_{!}( \pi_2 f  a_k), \pi_2 e'_k)   \, .
\]
Now, note that
\[
 \Id_A( \pi_1 f  a_k \, ,\;   \pi_1 e'_k )  = \Id_A(\pi_1 f a_k, a_k) \, , \quad 
 \Id_{E(a_k)}( p_{!}( \pi_2 f  a_k), \pi_2 e'_k)  = \Id_{E(a_k)}( sa_k, e_k) 
 \]
and that we have 
\begin{alignat*}{4}  
\overline{(\pi_1  f)}_k  & \iso  (\overline{\pi_1})_k \ct (\pi_1 \circ \bar{f}_k)  & & \qquad (\text{by definition of } \pi_1 f)   \\
& \iso  1_{a_k}  \ct (\pi_1 \circ \bar{f}_k)  & & \qquad (\text{by definition of } \pi_1)  \\
 & \iso (\pi_1 \circ \bar{f}_k) & & \qquad (\text{by the groupoid laws}) \\
 & \iso p  & & \qquad (\text{by definition of } \ext^\Sigma) \,  . \\
 \intertext{Therefore, we can construct the following chain of paths:}
p & \iso  \overline{(\pi_1 f)}_k & &  \qquad (\text{by what we just proved})  \\
  & \iso 1_{a_k} \ct  \alpha_{a_k}  & &  \qquad (\text{by the path } \bar{\alpha}_k ) \\ 
  & \iso  \alpha_{a_k}  & & \qquad (\text{by the groupoid laws}) \\
\intertext{Hence,  the required path $\bar{s}_k \co  \Id(s a_k,  e_k)$ can be defined as the following composite:}
s a_k & \deq (\alpha_{a_k})_{!} \big( \pi_2 f a_k   \big) & & \qquad (\text{by the definition of } s) \\
 &              \iso  p_{!} \big( \pi_2 f a_k   \big) & & \qquad (\text{since } p  \iso \alpha_{a_k}) \\
   &            \iso  e_k  & &  \qquad (\text{by the path } q)  \, .
   \end{alignat*} 
   This concludes the proof.
\end{proof}

The proof of Theorem~\ref{thm:bipointedmain} simplifies considerably within the extensional
type theory~$\Hext$ obtained by adding to $\Hint$ the identity reflection rule in~\eqref{equ:collapse}. In that type theory, 
there is a judgemental equality 
between the composite $\pi_1 f \co A \to A$ and the identity $1_A \co A \to A$, with which the
rest of the argument can be shortened considerably. In that setting, one obtains the familiar characterisation 
of an inductive type as strict initial algebras.

\medskip

Theorem~\ref{thm:bipointedmain} gives a logical equivalence between two types, but in fact we 
have a genuine equivalence of types, as the following corollary shows.

\begin{corollary} For a bipointed type $A$, there is an equivalence of types $\isbipind(A)\simeq   \isbiphinit(A)$.
\end{corollary} 

\begin{proof} Theorem~\ref{thm:bipointedmain} gives a logical equivalence, but  $\isbipind(A)$ is a 
mere proposition by Proposition~\ref{thm:isbipindisprop} and $\isbiphinit(A)$ is a mere proposition
by Proposition~\ref{thm:isbiphinitishprop}. 
\end{proof}

The next proposition characterizes the
type $\Bool$ up to equivalence. In its statement, we refer to the rules for $\Bool$ in
Table~\ref{tab:boolrules}.

\begin{corollary} Assuming the rules for the type $\Bool$, for a bipointed type $A = (A, a_0, a_1)$, the following 
conditions are equivalent:
\begin{enumerate}[(i)]
\item $A$ is inductive,
\item $A$ is homotopy-initial,
\item $A$ and $\Bool$ are equivalent as bipointed types.
\end{enumerate}
In particular, $\Bool$ is a homotopy-initial bipointed type. \qed
\end{corollary}

\begin{remark} \label{thm:smallelimination} Note that the elimination rules for $\Bool$ allow us to eliminate 
over an arbitrary, \ie not necessarily small, dependent type. Instead, the definition of an inductive bipointed
type involve the existence of sections over small fibered bipointed types. In spite of this apparent difference,
since $\Bool$ is assumed to be a small type, one can prove an equivalence between any inductive type $A$
and $\Bool$ and hence derive counterparts of the elimination rules for $\Bool$ for any inductive bipointed type. 

Let us point out that there are at least two alternatives to the approach taken here regarding universes. The first involves avoiding the restriction
to \emph{small} fibered bipointed types in the definition of the notion of an inductive bipointed type. Accordingly,
one drops the restriction of mapping into \emph{small} bipointed types in the definition of a notion of a homotopy-initial
algebra. With these changes, there is still a logical equivalence between the modified notions, but this is no longer an internal statement in the type theory, as in Theorem~\ref{thm:bipointedmain}. Alternatively, one could assume to
have a hiearchy of type universes $\U_0 \co \U_1 \co \; \ldots \; \co \U_n \co \U_{n+1} \co \ldots$ and modify the elimination rules for $\Bool$ by specifying that the types into which we are eliminating belong to some universe. 
A counterpart of Theorem~\ref{thm:bipointedmain}, now stated with appropriate universe levels, would still hold.
 \end{remark}

\subsection*{Univalence for bipointed types} \label{sec:unibip}
We conclude this section by showing that if the type universe~$\U$ is assumed to be univalent, then a form of the univalence axiom holds also for bipointed types, in the sense made precise by the next theorem, where we
use notation analogous to the one introduced for extension functions in Section~\ref{sec:bac}. This is an
instance of the Structure Identity Principle considered in~\cite{AczelP:stripu}.

\begin{theorem}  \label{thm:bipunivalence} Assuming the univalence axiom, 
for small bipointed types $A, B \co \Bip$, the  canonical function
\[ 
\ext^{\Bip}_{A,B} \co \Id_{\Bip} \big(A,B\big) \to  \BipEquiv(A,B) 
\] 
is an equivalence.
\end{theorem} 

\begin{proof} 
Let $ (A,a_0,a_1), (B,b_0,b_1)$ be small bipointed types. By the characterization of the identity types
of $\Sigma$-types, the 
identity type $\Id\big( (A,a_0,a_1),  (B,b_0,b_1)\big)$ is equivalent to  the type
\[
(\Sigma p \co \Id_\U(A,B))  \, \Id(( a_0,a_1 ),  p^* ( b_0,b_1))  \, .
\]
By $\Id$-elimination and the characterization of paths in product types, this type is equivalent to
\[ 
(\Sigma p \co \Id_\U(A,B)) \, \Id \big( (\ext \, p)(a_0),  b_0\big) \times \Id \big( (\ext \, p)(a_1) , b_1) \, ,
 \]
where $\ext \, p \co A \to B$ is the equivalence of types associated to $p \co \Id_\U(A,B)$.  By the univalence axiom,
the above type is equivalent to
\[ 
(\Sigma f \co \Eq(A,B)) \, \Id \big( f a_0 ,  b_0\big) \times \Id \big( f a_1 , b_1\big)  \, .
\]
After rearranging, we get
\[
(\Sigma f  \co A \to B)(\Sigma \bar{f}_0 \co \Id( fa_0, b_0)) (\Sigma \bar{f}_1 \co \Id( fa_1, b_1))  \, \isequiv(f) \, ,
\]
which is equivalent to $\BipEquiv(A,B)$ by Proposition~\ref{thm:usemere}. Finally, it is not hard to see that the composition of the above equivalences yields the  function $\ext^{\Bip}_{A,B}$ up to a homotopy, thus showing that it is an equivalence, as required.
\end{proof}

\begin{corollary} 
Assuming the univalence axiom, 
homotopy-initial small bipointed types are unique up to a contractible type of paths, i.e. the type
\[ 
(\Pi A \co \Bip) (\Pi B \co \Bip)
\big( \isbiphinit(A) \times \isbiphinit(B) \to \iscontr(\Id_\Bip(A,B)) \big) \, .
\] 
is inhabited.
\end{corollary}

\begin{proof} This is an immediate consequence of Proposition~\ref{BoolHInitIso} and 
Theorem~~\ref{thm:bipunivalence}. 
\end{proof}

\section{Polynomial functors and their algebras}
\label{section:wfiles}

\subsection*{Algebras and algebra morphisms}
The main aim of this paper is to carry out an analysis for well-ordering types (introduced in \cite{MartinLofP:intttp}), or W-types, analogous to the one we have just done for the type $\Bool$. We recall the rules for W-types in Table~\ref{tab:wrules}. There, we sometimes write $W$ for~$(\W x \co A) B(x)$ for
brevity.  Informally, a W-type can be seen as the free algebra for a signature
with arbitrarily many operations of possibly infinite arity, but no equations. The premises of the formation rule can be thought of as specifying a signature that has the elements of the type $A$ as~(names of) operations and in which the arity of~$a \co A$ is (the cardinality of) the type~$B(a)$. Then the introduction rule specifies the canonical way of forming an element of the free algebra, and the elimination rule can be seen as the propositions-as-types translation of the appropriate induction principle. As usual, the computation rule states what happens if we apply the
the elimination rule to a canonical element of the inductive type. Finally, we have a rule expressing the
closure of the type universe $\U$ under the formation of $W$-types.

\begin{table}[htb]
\fbox{\begin{minipage}{14.5cm}
\[
\begin{prooftree}
 A \co \type \qquad
 x \co A \vdash B(x) \co \type
 \justifies
 \textstyle
 (\W x \co A) B(x) \co \type
\end{prooftree} \qquad 
\begin{prooftree}
a \co A \qquad
t \co B(a) \to W
\justifies
\wsup(a,t)\co W
\end{prooftree}
\]
\bigskip
\[
\begin{prooftree}
w \co W \vdash E(w) \co \type  \quad
x \co A,\, u \co B(x) \to W,\, v \co (\Pi y \co B(x)) E(uy) \vdash e(x,u,v) \co E(\wsup(x,u))
\justifies
 w \co W \vdash \wind(w,e) \co E(w)
 \end{prooftree} \]
 \bigskip
 \[
\begin{prooftree}
w \co W \vdash E(w) \co \type  \quad
x\co A,\, u \co B(x) \to W,\, v \co (\Pi y \co B(x)) E(u y)  \vdash e(x,u,v) \co E(\wsup(x,u))
\justifies
x\co A,\, u \co B(x) \to W \vdash  \wind(\wsup(x,u),e) = e(x,u, ( \lambda y \co B(x)) \,\wind(u y,e)) \co E(\wsup(x,u))
\end{prooftree} \bigskip
\]
\[
\begin{prooftree}
 A \co \U \qquad
 x \co A \vdash B(x) \co \U
 \justifies
 \textstyle
 (\W x \co A) B(x) \co \U
\end{prooftree} \medskip
\]
\end{minipage}} \smallskip
\caption{Rules for $W$-types.} 
\label{tab:wrules}
\end{table}

We now consider a small type $A \co \U$ and a small dependent type $B \co A \to \U$, which we consider
fixed for this section and the next. For $C \co \U$,
we define
\[
PC \defeq (\Sigma x \co A) (B(x) \to C) \, .
\]
In this way, we obtain a function $P \co \U \to \U$. This operation on types extends to an operation on functions, 
as follows. For $f \co C \to D$, we define $P f  \co PC \to PD$ by $\Sigma$-elimination
so that, for~$x \co A$ and $u \co B(x) \to C$, we have 
\[
(Pf )((x, u)) = (x, f u) \, .
\] 
This assignment is pseudo-functorial in the sense that we have propositional, rather than judgemental, equalities:
\begin{equation}
\label{equ:pseudofunP}
\phi_{f, g} \co \Id( P(g \circ f), Pg \circ Pf) \, , \quad \phi_A \co \Id( P(1_A), 1_{PA})
\end{equation}
for $f \co C \to D$, $g \co D \to E$. We still refer to $P$ as the \emph{polynomial functor} associated to $A \co \U$ 
and $B \co A \to \U$, so as to highlight the analogy with the theory of polynomial functors on locally cartesian closed categories~\cite{GambinoN:weltdp,MoerdijkI:weltc}.

\begin{definition} A \emph{$P$-algebra} 
\[
(C, \sup_C)
\] 
is a small type $C \co \U$ equipped with a function~$\sup_C~\co~PC~\to~C$. 
\end{definition}

\smallskip

 The type of $P$-algebras is then defined as 
 \[
 \Palg  \defeq (\Sigma C \co \U) (PC \to C) \, .
 \]
 Given a $P$-algebra $C = (C, \sup_C)$,
 we refer to the type $C$ as the \emph{carrier} or \emph{underlying type} of the algebra and to the function $\sup_C \co PC\to C$ as the \emph{structure map} of the $P$-algebra. 
  In the
 presence of W-types, an example of $P$-algebra is given by the type $W \defeq (\W x\co A)B(x)$, with structure map given by the introduction rule for W-types. 
 
 \medskip
 
Let us now fix $P$-algebras $C = (C, \sup_C)$ and $D = (D, \sup_D)$. 

\begin{definition} A \emph{$P$-algebra morphism} $(f, \bar{f}) \co C \to D$ 
is a function $f \co C \rightarrow D$ equipped with a path $\bar{f} \co \Id( f \circ \sup_C \, , \sup_{D} \circ P f)$.
\end{definition}

Note that the homotopy associated to a path $\bar{f}$ as above has components
\[
(\ext \bar{f})_{x,u} \co \Id \big( f(\sup_C(x,u)), \sup_D(x, fu) \big) \, .
\]
A $P$-algebra morphism as above can be represented with a diagram of the form
\begin{equation*}
\vcenter{\hbox{\xymatrix@C=1.5cm{
 PC \ar[r]^{\sup_C} \ar[d]_{Pf}  \ar@{}[dr]|{\Downarrow \, \bar{f}} &  C \ar[d]^{f}\\
PD \ar[r]_{\sup_D}   & \; D \, .}}}
\end{equation*}
We use this slighly unconventional orientation of the diagram in order to stress the analogy with bipointed morphisms
({cf.\;}the diagram in~\eqref{equ:bipmor}).
Informally, one can think of the path $\bar{f}$ as a proof that the diagram commutes (which is the 
requirement defining the notion of morphism of endofunctor algebras in category theory) or as
an invertible 2-cell (as in the notion of a pseudo-morphism between algebras in 2-dimensional category 
theory~\cite{BlackwellR:twodmt}).  
For later use,  let us introduce some auxiliary notation. For $P$-algebras $C = (C, \sup_C)$, 
$D = (D, \sup_D)$ and a function $f \co C \to D$ between their underlying types, let us define
\begin{equation}
\label{equ:isalghom}
\mathsf{isalghom}(f) \defeq \Id( f \circ \sup_C \,  \sup_D \circ Pf) \, .
\end{equation}
Note that this type is not, in general, a mere proposition. 
Informally, $\mathsf{isalghom}(f)$ is the type of paths $\bar{f}$ witnessing that $f$ is a $P$-algebra morphism,
fitting in a diagram as above. Accordingly, the type of $P$-algebra 
morphisms between $C$ and $D$ is defined by
\[
\Palg(C,D)
 \defeq  
(\Sigma f:  C \rightarrow D) \, \mathsf{isalghom}(f) \, .
\]
We now define the composition operation for $P$-algebra morphisms. Given  $(f, \bar{f}) \co C \to D$ and~$(g, \bar{g}) \co D \to E$,
their composite $(gf, \overline{gf}) \co (C, \sup_C) \to (E, \sup_E)$
is obtained as follows. Its underlying function is given by $gf\co C \to E$, and so the 
the required path must be of the form
\[
 \overline{(g  f)} \co \Id\big( (g \circ f) \circ \sup_C   \, ,  \sup_E \circ P(g \circ f)  \big)\, .
\]
Such a path is obtained by pasting the diagrams 
\[
\xymatrix@C=1.5cm{
 PC \ar[r]^{\sup_C} \ar[d]_{Pf}  \ar@{}[dr]|{\Downarrow \, \bar{f}} &  C \ar[d]^{f}\\
PD \ar[r]_{\sup_D}  \ar[d]_{Pg}   \ar@{}[dr]|{\Downarrow \, \bar{g}} & D \ar[d]^g  \\
PE \ar[r]_{\sup_E} & \; E  \, .}
\] 
More precisely, it is given by the following composition of paths:
\[
\xymatrix@C=1.7cm{
g \circ f \circ \sup_C \ar[r]^{g \circ \bar{f}} & 
g \circ \sup_D \circ Pf \ar[r]^{\bar{g} \circ Pf} & 
\sup_E \circ Pg \circ Pf \ar[r]^{\sup_E \circ \phi_{f,g}^{-1}} &
\sup_E \circ P(g \circ f) \, ,}
\]
where we used the pseudo-functoriality of $P$ in~\eqref{equ:pseudofunP}. 
For a $P$-algebra $C$,  the identity function $1_C \co C \to C$ has an evident structure of $P$-algebra morphism,
represented in the diagram
\begin{equation}
\label{equ:palgid}
\xymatrix@C=1.5cm{
PC \ar[d]_{P(1_C)}  \ar[r]^{\sup_C} \ar@{}[dr]|{\Downarrow \, \overline{1}_C} & C \ar[d]^{1_C} \\
PC \ar[r]_{\sup_C} & \; C \, .}
\end{equation}
As in the case of bipointed types, the associativity and unit laws for a category do not hold up to judgemental equality, but only so up to a system of higher and higher paths.

\medskip

We will require an alternative description of the identity type between two $P$-algebra morphisms. For this, we introduce  the notion of a $P$-algebra homotopy in the next definition.

\medskip

Let us fix $P$-algebra morphisms $f = (f, \bar{f})$ and $g =  (g, \bar{g})$ from $C$ to $D$.

\begin{definition} A \emph{$P$-algebra homotopy}  $(\alpha, \bar{\alpha}) \co f \to g$
is a homotopy $\alpha \co\Hot( f , g)$ equipped with a homotopy
$\bar{\alpha} \co 
\Hot \big(  ( \sup_D \circ P \alpha ) \ct  (\ext \, \bar{f})  \, , \;  (\ext \, \bar{g})  \ct (\alpha \circ \sup_C)\big)$.
\end{definition}

Note that in the definition $\alpha \circ \sup_C$ and $\sup_D \circ P \alpha$ are obtained by pre-composition and post-compositions, respectively, of functions with homotopies. The homotopy $\bar{\alpha}$ can be thought of as a proof that the two homotopies produced by the pasting diagrams
\[
\xymatrix@C=2cm@R=1.5cm{
PC \ar[r]^{\sup_C} \ar@/_1.3pc/[d]_{Pg}  \ar@/^1.3pc/[d]^{Pf}  \ar@{}[dr]|{\qquad \Downarrow \, \bar{f}}
\ar@{}[d]|{\Downarrow \, P\alpha}
  & C \ar@/^1.3pc/[d]^{f} \\ 
PD \ar[r]_{\sup_D} & D} \qquad
\xymatrix@C=2cm@R=1.5cm{
PC \ar[r]^{\sup_C} \ar@/_1.3pc/[d]_{Pg}  \ar@{}[dr]|{\Downarrow \, \bar{g} \qquad} & C \ar@/^1.3pc/[d]^{f} \ar@/_1.3pc/[d]_{g} \ar@{}[d]|{\Downarrow \, \alpha} \\ 
PD \ar[r]_{\sup_D} & D} 
\]
are equal, which is analogous to the condition defining an algebra 2-cell in 2-dimensional category 
theory~\cite{BlackwellR:twodmt}.   Explicitly, the
component of~$\bar{\alpha}$ associated to $x \co A$ and $u \co B(x) \to C$ fit into diagrams of the form
\[
\xymatrix@C=1.8cm{
f (\sup_C(x,u)) 
\ar[r]^{(\ext \bar{f})_{x,u}}
\ar[d]_{\alpha_{\sup_C(x,u)}} 
\ar@{}[dr]|{\Downarrow \, \bar{\alpha}_{x,u}}  & 
 \sup_D(x, fu) \ar[d]^{\sup_D(x, \int( \alpha_u))} \\
g (\sup_D(x,u)) \ar[r]_{(\ext \bar{g})_{x,u}} &\;  \sup_D(x, gu) \, ,}
\]
where $\int(\alpha_u)$ denotes the path associated to the homotopy 
$(\lambda y \co B(x)) \alpha_{uy}$ between $fu$ and $gu$.
The type of $P$-algebra homotopies is then defined by
\[
\AlgHot \big( (f,\bar{f}), (g, \bar{g})  \big)
 \defeq  
(\Sigma \alpha \co \Hot(  f , g)) \, \Id\big( 
( \sup_D \circ P \alpha ) \ct  (\ext \, \bar{f}) \, , \;  ( (\ext \, \bar{g}) \circ \sup_C) \cdot (\alpha \circ \sup_C)
 \big) \, .
\]

\newcommand{\HotAlg}{\mathsf{HotAlg}}

\begin{lemma}\label{IdEqHo}
For every pair of $P$-algebra morphisms $(f, \bar{f}) \, , (g, \bar{g}) \co C \to D$,  
the canonical function
\[
\ext^{\Palg}_{f,g}  \co 
\Id\big((f, \bar{f}), (g, \bar{g})\big) \to \HotAlg \big((f, \bar{f}), (g, \bar{g})\big).
\]
 is an equivalence of types. 
\end{lemma}

\begin{proof}
This follows from a more general statement to be proved in Lemma~\ref{lem:fibhomeqid} below.
\end{proof}
This is another case of the identity type encoding higher-categorical structure; we note that the proof of Lemma \ref{IdEqHo}
does not require the univalence axiom.

\subsection*{Fibered algebras and algebra sections} We now introduce the fibered versions of the notions of a $P$-algebra, $P$-algebra morphism, and $P$-algebra homotopy. Some preliminary remarks will help us to motivate our definitions. Let us consider a fixed $P$-algebra   $C = (C, \sup_C)$.
Given a dependent type $E \co C \to \U$,
we wish to describe what data determines a $P$-algebra structure on the type $E' \defeq (\Sigma z \co C) E(z)$. First of all, using a special case of the $\Pi\Sigma$-distributivity law recalled in~\eqref{equ:ac}, we have 
\[
PE' \simeq (\Sigma x \co A) (\Sigma u \co B(x) \to C) (\Pi y \co B(x)) E (uy) \, .
\]
Therefore, we obtain
\begin{align*} 
PE' \to E' \ & \simeq 
 (\Sigma x \co A) (\Sigma u \co B(x) \to C) (\Pi y \co B(x)) E (uy)  \to (\Sigma z \co C) E(z) \\ 
 & \simeq (\Pi x \co A)(\Pi u \co B(x) \to C) (\Pi v \co (\Pi y \co B(x)) E(uy)) (\Sigma z \co C) E(z) \, , 
\end{align*}
and so a structure map $\sup_{E'} \co PE' \to E'$ can be viewed equivalently as a function which 
takes arguments $x \co A$, $u \co B(x) \to C$, $v \co (\Pi y \co B(x))E(uy)$ and an element of $E'$.
Thus, if we wish to ensure that the structure map $\sup_{E'} \co PE' \to E'$ is such that the projection function
 $\pi_1 \co E' \to C$ is a $P$-algebra morphism,  \ie that we can find a path fitting in the diagram
 \[
 \xymatrix@C=1.5cm{
 PE' \ar[d]_-{P \pi_1} \ar[r]^{\sup_{E'}} \ar@{}[dr]|{\Downarrow \,  \overline{\pi}_1} & E' \ar[d]^{\pi_1} \\
 PC \ar[r]_{\sup_C} & \; C \, ,}
 \]
it is sufficient to require the existence of a function of the form
\[
e \co (\Pi x \co A)(\Pi u \co B(x) \to C) (\Pi v \co (\Pi y \co B(x)) E(uy))  E(\sup_C(x,u)) \, . 
 \]  
Note that such a function appears also in one of the premisses of the elimination rule for $W$-types in Table~\ref{tab:wrules}.  We are therefore led to make the following definition. 

\begin{definition} \label{def:fibalg}
A \emph{fibered $P$-algebra} over $C$ consists of a dependent type $E \co C \to \U$
and a function  $e \co  (\Pi x \co A) (\Pi u \co B(x) \to C)   ((\Pi y \co B(x))   E(u y))  \,  E(\sup_C(x,u))$.
\end{definition}

We define the type of fibered $P$-algebras 
over $C$ as follows:
\[
\FibPalg(C) \defeq (\Sigma E \co C \to \U) (\Pi x \co A) (\Pi u \co B(x) \to C) 
 ((\Pi y \co B(x)) E(u y))\,  E(\sup_C(x,u))
 \]
  
 Let us consider a fixed fibered $P$-algebra $E = (E, e)$ over $C$. 
 
 \medskip
 
We define the $P$-algebra $E' = (E', \sup_{E'})$, to which we shall refer as the \emph{$P$-algebra associated to $E$}, 
as follows. As before,
we define $E' \defeq (\Sigma z \co C) E(z)$ and  $\sup_{E'} \co PE' \to E'$  by $\Sigma$-elimination 
so that, for $x \co A$ and~$u \co B(x) \to E'$, we have
\[
\sup_{E'}(x,u) = 
\pair \big( 
\sup_C(x, \pi_1 u ) \, , 
e( x, \pi_1 u, \pi_2 u) \big)  \, .
\]
Here, note that $\pi_1 u \co B(x) \to C$ and $\pi_2 u \co (\Pi y \co B(x)) E(\pi_1 u y)$ and so,
by the type of $e$, we have that~$e(x, \pi_1 u, \pi_2 u) \co E( \sup_C(x, \pi_1 u))$, as required.

In analogy with the way we defined fibered $P$-algebras, it is possible to define 
$P$-algebra sections. To state this definition, we need some preliminary notation.
For  $f \co (\Pi z \co C) E(z)$,
we define $e_f \co (\Pi x \co A)(\Pi u \co B(x) \to C) E(\sup_C(x,u))$ by letting
\begin{equation}
\label{equ:ef}
 e_f \defeq (\lambda x \co A)(\lambda u \co B(x) \to C) \, e(x, u, f u) \, .
\end{equation}
Here, note that for $y \co B(x)$, we have $u y \co C$ and hence $f u y \co E(uy)$,
as required.

\begin{definition} \label{def:fibalgsection} 
A \emph{$P$-algebra section} $(f, \bar{f})$ of $E$ is a section~$f \co (\Pi z \co C) E(z)$ equipped
with a path 
$\bar{f} \co  \Id \big( f \,  \sup_C \, , e_f \big)$.
\end{definition} 

Note that the components of the homotopy $\ext \, \bar{f}$ associated to a path $\bar{f}$ as above have the 
form $(\ext \, \bar{f})_{x,u} \co \Id \big( f(\sup_C(x,u)) \, , e(x, u, fu) \big)$.   We define the type of $P$-algebra sections of $E$ by letting
\[ 
\PalgSec(C,E)  \defeq (\Sigma f  \co (\Pi x \co C) E(x)) \; \Id \big( f \, \sup_C \, , e_f  \big) \, .
\]
This terminology is justified by the fact that, given an algebra section $(f, \bar{f})$ of $E$, there is an algebra 
morphism $f' \co C \to E'$, where  $E' = (E', \sup_{E'})$ is  the $P$-algebra associated to $E$. 
Its underlying function is defined by letting $f' \defeq (\lambda z \co C) \, \pair(z, fz)$
and direct calculations show that there is a path fitting in the diagram
\[
\xymatrix@C=2cm{
PC \ar[d]_{Pf'} \ar[r]^{\sup_C} \ar@{}[dr]|{\Downarrow \, \overline{f'}} & C \ar[d]^{f'} \\
PE' \ar[r]_{\sup_{E'} } & \; E' \, .}
\]
This $P$-algebra morphism provides
a section of the $P$-algebra morphism~$\pi_1 \co E' \to C$ in the sense that the composite $P$-algebra
morphism~$\pi_1  f' \co C \to C$ can be shown to be propositionally equal to the identity $P$-algebra
morphism~$1_C \co C \to C$.

\medskip

We will require an analysis of paths between of $P$-algebra sections and thus we introduce, in Definition~\ref{def:W2cellsection} below, the
notion of a homotopy between $P$-algebra sections. In order to state the definition more briefly, let us introduce some
notation. For a fibered $P$-algebra $E = (E,e)$, sections~$f, g \co (\Pi z \co C) E(z)$ and a path $p \co
\Id(f,g)$, we write $e_p \co \Id(e_f, e_g)$ for the evident path defined by $\Id$-elimination, 
where~$e_f$ and~$e_g$ are defined as in~\eqref{equ:ef}. By the characterisation of identity types of
function types,  a homotopy $\alpha \co \Hot(f,g)$ determines also 
a homotopy~$e_\alpha \co \Hot(e_f, e_g)$. For $x \co A$ and $u \co B(x) \to C$, the 
component $(e_\alpha)_{x,u}$ of this homotopy is given by $e(x, u, \int(\alpha_u))$, where
$\int(\alpha_u)$ is the path associated to the homotopy $(\lambda y \co B(x)) \alpha_{uy}$.

\medskip

Let us now fix two $P$-algebra sections  of $E$,  $f = (f, \bar{f})$ and $g = (g, \bar{g})$. 

\begin{definition} \label{def:W2cellsection}  A \emph{$P$-algebra section homotopy} 
$(\alpha , \bar{\alpha}) \co f  \sim g$  
is a homotopy~$\alpha \co \Hot(f, g)$ equipped with a homotopy
$\bar{\alpha} \co 
(\Pi x \co A) (\Pi u \co B(x) \to C) \;
\Hot\big(  e_\alpha \ct \ext( \bar{f}) \, , 
 \ext( \bar{g}) \ct (\alpha \circ \sup_C)   \big)$.
\end{definition} 

The components of the homotopy $\bar{\alpha}$ that is part of a $P$-algebra section homotopy as above  can be
represented diagrammatically as fitting in the following diagram
\[
\xymatrix@C=2.5cm{
f ( \sup_C(x,u)) \ar[d]_{\alpha_{\sup_C(x,u)}} \ar[r]^{(\ext \, \bar{f})_{x,u}}  
\ar@{}[dr]|{\Downarrow \, \bar{\alpha}_{x,u}}  & e(x,u, fu)  \ar[d]^{e(x,u, \int(\alpha_u))} \\ 
g(\sup_C(x,u)) ) \ar[r]_{(\ext \, \bar{g})_{x,u}} & \;  e(x, u, gu) \, .}
\]
Accordingly, we define the type of $P$-algebra homotopies of sections as follows:
\begin{multline*} 
\AlgSecHot( (f, \bar{f}) ,\, (g, \bar{g}) )  \defeq  \\ 
(\Sigma \alpha \co \Hot( f , g)) 
(\Pi x \co A) 
(\Pi u \co B(x) \to C) \, 
\Hot\big(  e_\alpha \ct \ext( \bar{f}) \, , 
 \ext( \bar{g}) \ct (\alpha \circ \sup_C)   \big) \, .
\end{multline*}
Remarkably, in spite of the complexity of its definition, the notion of a $P$-algebra homotopy is equivalent to
that of an identity proof between $P$-algebra sections, as the next lemma makes precise.

\begin{lemma}\label{lem:fibhomeqid} The canonical function
\[
\ext^\PalgSec_{(f, \bar{f}), (g, \bar{g})} \co \Id \big( (f, \bar{f}) ,\, (g, \bar{g}) \big) \, \to
\AlgSecHot\big( (f, \bar{f}) ,\, (g, \bar{g}) \big) 
\]
is an equivalence of types.
\end{lemma}

\begin{proof} For $p \co \Id(f, g)$ we have  $p_{!}(\bar{f}) \co  \Id( g \circ \sup_C \, , e_g)$ 
and it can be shown by $\Id$-elimination that there exists a path 
$q  \co \Id\big(  e_p \ct \bar{f}   \, ,  p_{!}(\bar{f})  \ct  (p \circ \sup_C)    \big)$, 
which can be represented with the diagram
\[
\xymatrix@C=2cm{
f \circ \sup_C  \ar[d]_-{p \circ \sup_C}  \ar[r]^{\bar{f}}  \ar@{}[dr]|{\Downarrow \, q } & e_f 
\ar[d]^{ e_p}  \\
 g \circ \sup_C \ar[r]_-{ p_{!}(\bar{f}) }  & \; e_g \, .    }
\]
We then have
\begin{align*}
  \Id\big( (f,\bar{f}),  (g,\bar{g}) \big) 
& \simeq (\Sigma p \co \Id( f , g)) \; \Id ( p_{!} ( \bar{f} ) \, , \bar{g}  )  \\
& \simeq (\Sigma p \co \Id( f, g)) 
\; \Id\big(  
e_p \ct  \bar{f} \ct (p \circ \sup_C)^{-1} \, , \bar{g}     \big)   \\ 
& \simeq (\Sigma p \co \Id( f, g)) 
\; \Id\big( e_p  \ct \bar{f} \, ,   \bar{g} \ct ( p \circ \sup_C)   \big)   \\ 
& \simeq (\Sigma p \co \Id( f, g)) 
\Hot\big(  e_{\ext p} \ct \ext( \bar{f}) \, , 
 \ext( \bar{g}) \ct ( (\ext \, p) \circ \sup_C)   \big)  \\ 
&  \simeq (\Sigma \alpha \co \Hot( f, g)) \; 
\Hot\big(  e_\alpha \ct \ext( \bar{f}) \, , 
 \ext( \bar{g}) \ct (\alpha \circ \sup_C)   \big)  \\
&  =   \AlgSecHot \big( (f,\bar{f}) \; (g,\bar{g}) \big) \, . \qedhere
\end{align*}  
\end{proof}

Note that Lemma~\ref{IdEqHo}, which we left without proof, follows as a special case of Lemma~\ref{lem:fibhomeqid}. 

\subsection*{Algebra equivalences}  We introduce the notion of equivalence between $P$-algebras. This
will useful in Section~\ref{sec:univalencealgebras}, where will prove that assuming the univalence axiom, a form of univalence holds also for $P$-algebras.

\begin{definition}  
We say that a $P$-algebra morphism $f \co C \to D$ is 
 a~\emph{$P$-algebra equivalence}
if there exist $P$-algebra morphisms $g, h \co D \to C$  which provide a left and a right $P$-inverse for $f$ as a
$P$-algebra morphism, \ie for
which there are paths of $P$-algebra morphisms
\[ 
p \co \Id_{\Palg(C,C)}( g  f,  1_C) \, , \quad q \co \Id_{\Palg(D,D)}( f h , 1_D) \, .
\]
\end{definition}
 
Given a $P$-algebra morphism $f \co C \to D$, we define the type of proofs that $f$ is an equivalence of $P$-algebras as follows:
\[
\isalgequiv(f) \defeq  (\Sigma g \co  \Palg(D,C)) \Id_{\Palg(C,C)}( g f, 1_C )  \times 
    (\Sigma h  \co \Palg(D, C)) \Id_{\Palg(D,D)}( f h , 1_D ) \, .
\]
We then define the type of $P$-algebra equivalences between $C$ and $D$ as
\[
\AlgEquiv(C, D)
\defeq   (\Sigma f \co \Palg(C,D)) \, \isalgequiv(f)  \, . 
\]

\begin{lemma} The underlying function of a $P$-algebra equivalence is an equivalence, \ie 
for every $P$-algebra morphism $(f, \bar{f}) \co C \to D$ there is a function 
\[
\pi_f \co \isalgequiv(f, \bar{f})  \to \mathsf{isequiv}(f)  \, .
\]
\end{lemma}

\begin{proof} Let $(f, \bar{f}) \co C \to D$ be a $P$-algebra morphism. 
Unfolding the definition, we have
\begin{multline*}
\isalgequiv(f, \bar{f}) \defeq \\	 (\Sigma g \co D \to C) 
	 \big( \Sigma \bar{g} \co \mathsf{isalghom}(g) \big) \, 
 		G(g, \bar{g})  \times
	 (\Sigma h \co D \to C) 
	 \big(\Sigma \bar{h} \co   \mathsf{isalghom}(h) \big)
	 \,  H(h, \bar{h}) \, , 
\end{multline*}
where we used the notation introduced in~\eqref{equ:isalghom} and 
\[
G(g,\bar{g}) 
\defeq 
\Id \big( ( g  f,  \overline{gf}), (1_C, \overline{1}_C) \big)  \, , \quad 
H(h,\bar{h})     \defeq \Id \big( (   f h,  \overline{ f h}), (1_D, \overline{1}_D) \big)  \, .
\]
The types $G(g, \bar{g})$ and $H(h, \bar{h})$ can be thought of as the types of proofs that $(g, \bar{g})$ and $(h, \bar{h})$ are a left and right inverse for $(f, \bar{f})$ as 
$P$-algebra morphisms, respectively. For the right inverse, this amounts to requiring that the pasting diagram
\[
\xymatrix@C=1.5cm{
PC \ar[r]^{\sup_C} \ar[d]_{Pf}  \ar@{}[dr]|{\Downarrow \, \bar{f}} & C \ar[d]^f \\
PD \ar[r]^{\sup_D} \ar[d]_{Pg} \ar@{}[dr]|{\Downarrow \, \bar{g}} & D\ar[d]^g  \\
PC \ar[r]_{\sup_C} & C}
\]
is propositionally equal to the diagram for the identity $P$-algebra morphism on $C$ in~\eqref{equ:palgid}. 
By the characterization of paths in $\Sigma$-types, we have
\[ 
G(g, \bar{g})  \simeq (\Sigma p \co \Id(g f ,  1_C) ) \, 
	\Id\big( \overline{gf}   ,\, p^*( \overline{1}_C)  \big) \, , \quad
H(h, \bar{h})  \simeq (\Sigma q \co \Id( f h ,  1_C) ) \, 
	\Id\big( \overline{fh}   ,\, q^*( \overline{1}_C)  \big) \,. 
\]
Thus, rearranging the $\Sigma$-types in the definition, we have
\begin{multline}
\label{equ:isalgequivequiv}
\isalgequiv(f, \bar{f}) \simeq \\ 
 (\Sigma g \co D \to C)  (\Sigma p \co \Id(g f ,  1_C) ) \, G'(g,p) \times
   (\Sigma h \co D \to C)   (\Sigma q \co \Id( f h ,  1_C) )\,  H'(h,q) \, ,
\end{multline}
where 
\begin{align}
G'(g,p)  & \defeq (\Sigma \bar{g} \co \mathsf{isalghom}(g)) \; \Id( \overline{gf}, p^*(\bar{1}_C)) \, , \label{equ:alggp} \\
H'(h,q)  & \defeq (\Sigma \bar{h} \co \mathsf{isalghom}(h)) \; \Id( \overline{fh}, q^*(\bar{1}_D)) ) \, .  \label{equ:alghq}
\end{align}
The  canonical elements of~$G(g,p)$ are pairs $(\bar{g}, \bar{p})$ consisting of  a path~$\bar{g}$ making
$g$ into a $P$-algebra morphism and a path~$\bar{p}$ making $p \co \Id(gf, 1_C)$ into a propositional
equality between the~$P$-algebra morphisms~$(gf, \overline{gf})$ and~$(1_C, \bar{1}_C)$. It is now
clear that we can obtain the required function $\pi_f$ by composing the equivalence in~\eqref{equ:isalgequivequiv}
with the evident projections and the equivalence in~\eqref{equ:equivalternative}. 
\end{proof} 

Proposition~\ref{WAlgSpace} below can be understood informally as saying that for a $P$-algebra morphism $f$, there is an essentially unique way of turning an inverse of $f$ as a function into an inverse of $f$ as a $P$-algebra morphism.

\begin{proposition}\label{WAlgSpace} A $P$-algebra morphism $(f, \bar{f}) \co C \to D$ is an equivalence of
$P$-algebras if and only
if its underlying function $f \co C \to D$ is an equivalence of types, \ie the function 
\[
\pi_f \co \isalgequiv(f, \bar{f})  \to \mathsf{isequiv}(f)  
\]
is an equivalence. 
\end{proposition}

\begin{proof} Let $(f, \bar{f}) \co (C, \sup_C) \to (D, \sup_D)$ be a  $P$-algebra morphism. We will 
show that all the homotopy fibers of the function $\pi_f$ are contractible.
So, let us consider a canonical element of the codomain of $\pi_f$,  given by a 4-tuple $( g, h, p, q) \co \mathsf{isequiv}(f)$ consisting of functions $g \co D \to C$ and $h \co D \to C$ and paths $p \co \Id(gf, 1_C)$, $q \co \Id(fh, 1_D)$,
exhibiting~$g$ and~$h$ as a right and a left inverse of $f$ (as a function, not as a $P$-algebra morphism), respectively. 

The homotopy fiber of $\pi_f$ over this element can be thought of as the type consisting of all
the data that is missing from having a left and a right inverse of $f$ as a $P$-algebra morphism. 
In particular, we have
\begin{equation*}
\hfiber(\pi_f, (g, h, p, q ) ) \simeq G'(g,p) \times H'(h,q) \, ,
\end{equation*}
where $G'(g,p)$ and $H'(h,q)$ are defined as in~\eqref{equ:alggp} and~\eqref{equ:alghq}, respectively.
Therefore, it suffices to prove that $G'(g,p)$ and $H'(h,q)$ are  contractible.  The proofs that $G'(g,p)$ and $H'(h,q)$ are contractible are essentially identical, so we consider only $G'(h,p)$.

First of all,  recall that the path $\overline{gf} \co  \mathsf{isalghom}(gf) $ is given by 
$\overline{gf} \defeq (g \circ \bar{f}) \ct (\bar{g} \circ Pf )$, 
where we suppressed the path relative to the pseudo-functoriality of $P$, as in~\eqref{equ:pseudofunP}, for convenience. 
By $\Id$-elimination on $p$, the path $p^*(\overline{1}_C) \co \mathsf{isalghom}(gf)$ is propositionally equal to the composite path
\[
\xymatrix@C=1.7cm{
\sup_C \circ P(gf) \ar[r]^{\sup_C \circ P(p)} & \sup_C \circ P(1_C) \ar[r]^{\bar{1}_C} & 1_C \circ \sup_C \ar[r]^{p^{-1} \circ \sup_C} & 
(g f) \circ 1_C \, .}
\]
Hence, we have 
\begin{align*} 
G(g,p)  & \simeq  (\Sigma \bar{g} \co  \mathsf{AlgHom}(g))  \; 
\Id \big( (g \circ \bar{f}) \ct (\bar{g} \circ Pf ) \, , 
  (p^{-1} \circ \sup_C) \ct \bar{1}_C \ct (\sup_C \circ P(p)) \big) \\
& \simeq 
(\Sigma \bar{g} \co  \mathsf{AlgHom}(g))  \; 
\Id \big( \bar{g} \circ Pf  \, , (g \circ \bar{f})^{-1} \ct 
  (p^{-1} \circ \sup_C) \ct \bar{1}_C \ct (\sup_C \circ P(p)) \big) \, .
\end{align*} 
Now, since $f \co C \to D$ is an equivalence, $Pf \co PC \to PD$ is also an equivalence and hence so
is the function mapping a path $r \co \Id_{PD \to C}(s,t)$ to the composite $r \circ Pf \co \Id_{PC \to C}( s \circ Pf, t \circ Pf)$. Thus, by part~(ii) of Lemma~\ref{thm:useful}, $G'(g,p)$ is contractible, as required.
\end{proof}

\begin{corollary} For every $P$-algebra morphism $(f, \bar{f})$, the type $\isalgequiv(f, \bar{f})$ is
a mere proposition. \qed
\end{corollary}

\section{Homotopy-initial algebras}
\label{sec:homta}

\subsection*{Inductive algebras}
Given a $P$-algebra $C = (C, \sup_C)$ and a type $D$, an equivalence of types $f \co C \to D$ makes $D$ into a $P$-algebra with structure map $\sup_D \co PD \to D$ given by the composite
\[
\xymatrix@C=1.2cm{
PD \ar[r]^{P(f^{-1})} & PC \ar[r]^{\sup_C} & C \ar[r]^f & D \, , }
\]
 where $f^{-1} \co D \to C$ is a quasi-inverse of $f \co C \to D$. In particular, for $W = (\W x \co A) B(x)$, if we have an equivalence  
 $f \co W \to D$, then the induced $P$-algebra structure $\sup_D \co PD \to D$ defined as above is such that $D$ also satisfies 
 a form of the elimination rule for $W$-types.  We shall see that $D$ satisfies the other rules as well, but with a weakened computation rule.

\begin{definition}\label{def:Wind}
We say that a $P$-algebra $C$  is \emph{inductive} if every fibered $P$-algebra
over it has a $P$-algebra section, \ie the type
\[ 
\isalgind(C) \defeq (\Pi E \co \FibPalg(C)) \,  \PalgSec(C,E) 
\]  
is inhabited.
\end{definition}

In complete analogy with the case of bipointed types, for a $P$-algebra $C$, the type $\isalgind(C)$ is a mere proposition. We also have the following analogue of Proposition~\ref{BoolHInitIso}.

\begin{proposition} \label{WHInitIso} 
Homotopy-initial $P$-algebras are unique up to a  contractible type of algebra equivalences, i.e. the type
\[ 
(\Pi C \co \Palg) (\Pi D \co \Palg) \big( \isalghinit(C) \times \isalghinit(D)  \to 
\iscontr(\AlgEquiv(C,D)) \big) \, .
\] 
is inhabited.
\end{proposition}

\begin{proof} Let $C$ and $D$ be $P$-algebras. The type $\Palg(C,D)$ is contractible by homotopy-initiality of $C$. Since the dependent sum of a family of mere propositions over a mere proposition is again a mere proposition, it suffices to prove $\iscontr(\isalgequiv(f))$ for any $P$-algebra morphism $f$. This type is a mere proposition, as remarked earlier; thus it suffices to show it is inhabited.
Since $D$ is homotopy-initial, there exists a $P$-algebra morphism $g \co D \to C$. Again by homotopy-initiality of $C$ and 
$D$, we have $\Id(g \circ f, 1_C)$ and $\Id(f \circ g, 1_D)$, which gives us the desired $P$-algebra equivalence between 
$C$ and $D$.
\end{proof}

 The next proposition characterizes
inductive $P$-algebras by means of deduction rules, where we display premisses in multiple lines for
lack of space.

\begin{proposition} \label{thm:palgindrec}
Let $C = (C, \sup_C)$ be a $P$-algebra. Then $C$ is inductive if and only if it satisfies the following rules:

\smallskip

\begin{enumerate}[(i)]
\item the elimination rule, 
\[
\begin{prooftree}
\begin{array}{rcl}
z \co C &  \vdash & E(z) \co \U \\ 
\textstyle
x \co A \, , u  \co B(x) \to C,\, v \co (\Pi y \co B(x))  E(u y) &  \vdash & e(x,u,v) \co E(\sup_C(x,u))
\end{array}
\justifies
z \co C \vdash \elim(z,e) \co E(z)
\end{prooftree}
\]

\bigskip

\item the computation rule,
\[
\begin{prooftree}
\begin{array}{rcl}  
z \co C & \vdash & E(z) \co \U \\ 
\textstyle
x \co A ,\, u \co  B(x) \to C ,\, v \co (\Pi y \co B(x))  E(uy) &  \vdash & e(x,u,v) \co E(\sup_C(x,u))
\end{array}
\justifies
x \co A,\, u \co B(x) \to C
   \vdash 
   \comp(x,u,e) \co
    \Id \big( \elim(\sup_C(x,u),e),  e(x,u, (\lambda y \co B(x)) \elim(u y, e)) \big)\,.
\end{prooftree}
\]
\end{enumerate}
\end{proposition}

\begin{proof} The rules are simply an unfolding of the definition of an inductive algebra.
\end{proof}

\medskip

Below, when working with an inductive $P$-algebra, we will always assume to have constants $\elim$ and $\comp$ as in Proposition~\ref{thm:palgindrec}. We now show  the essential uniqueness of algebra sections of inductive
fibered algebras. 

\begin{proposition} \label{lem:Wetaind}
Let $C = (C, \sup_C)$ be a $P$-algebra. If $C$ is inductive, then it satisfies the following rules:

\begin{enumerate}[(i)]
\item the $\eta$-rule, 
\[
\begin{prooftree}
\begin{array}{rcl} 
 z \co C & \vdash  & E(z) \co \U   \\ 
 \textstyle x \co A,\,    u \co B(x) \to C,\, e \co (\Pi y \co B(x)) E(uy ) & \vdash &  e(x,u,v) \co E(\sup_C(x,u))  \\  
  z \co C & \vdash  & f(z) \co E(z) \\ 
 x \co A \, ,  u \co B(x) \to C  & \vdash  & \phi_{x,u} \co \Id  \big(  f(\sup_C(x,u)) ,  e\big(x,u, f  u ) \big) 
 \end{array}
 \justifies
z \co C \vdash \eta_z \co \Id( f(z),  \elim(z,e))
\end{prooftree}
\]

\bigskip

\item the coherence rule,
\[
\begin{prooftree}
\begin{array}{rcl}
z \co C & \vdash &  E(z) \co \U   \\ 
\textstyle x \co A,\,    u \co B(x) \to C,\, v \co (\Pi y \co B(x))E(uy) &  \vdash & e(x,u,v) \co E(\sup_C(x,u))  \\  
 z \co C &  \vdash & f(z) \co E(z) \\ 
x \co A \, ,  u \co B(x) \to C & \vdash  & \phi_{x,u} \co \Id  \big(  f(\sup_C(x,u)) ,  e\big(x,u, f  u ) \big) 
\end{array}
\justifies
x \co A, u \co B(x) \to C \vdash \bar{\eta}_{x,u} \co
\Id\big( \eta_{\sup_C(x,u)} \ct \comp(x,u,e), \; 
 \phi_{x,u} \cdot e(x,u,\int(\eta_u)) \big)
\end{prooftree}
\]

\end{enumerate}

\end{proposition}

Before proving the proposition, observe that the paths $\bar{\eta}_{x,u}$ in the conclusion of the
coherence rule can be seen as fitting in the
diagram
\[
\xymatrix@C=3cm{
f(\sup_C(x,u)) \ar[r]^{\eta_{\sup_C(x,u)}} \ar[d]_{\phi_{x,u}} 
\ar@{}[dr]|{\Downarrow \, \bar{\eta}_{x,u}} & \elim(\sup_C(x,u), e))    \ar[d]^{\comp(x,u,e)} \\
 e(x,u,fu) \ar[r]_-{e(x,u, \int( \eta_u))} & e(x, u, \elim(x,u, (\lambda y \co B(x)) \, \elim(uy, e) )) }
\]

\begin{proof}[Proof of Proposition~\ref{lem:Wetaind}]  For $z \co C$, let us define $T(z) \defeq \Id \big(   f(z), \elim(z,e))$. With this notation, proving
 the $\eta$-rule amounts to defining $\eta_z \co T(z)$, for $z \co C$. In order to do so, we apply
the elimination rule for $C$. We need to show that, for $x \co A$, $u \co B(x) \to C$ and~$v \co 
(\Pi y \co B(x)) \, T(uy)$, there 
is 
\[
t(x,u,v) \co   T(\sup_C(x,u)) \, .
\]
Note that $(\lambda y \co B(x)) \, v_{uy}$ is a homotopy between $fu$ and $(\lambda y \co B(x)) \, \elim(uy, e)$.
Hence, we have a corresponding path $\int(v_u)$. We can construct the required path  as follows:
\begin{align*}
f(\sup_C(x,u)) &\iso e\big(x,u , f u \big)   \by{\phi_{x,u}}\\
	&\iso e\big(x,u, (\lambda y \co B(x) )\, \elim(u y ,e) \big) \by{\int( v_u)} \\
	& \iso \elim(\sup_C(x,u),e) \by{\comp(x,u,e)^{-1}}.
\end{align*}
For $z \co C$, we can then define
\[
\eta_z \defeq \elim(z,t) \, .
\] 
For $x \co A$ and $u \co B(x) \to C$, the  computation rule of Proposition~\ref{thm:palgindrec} then gives us
\[
 \eta_{\sup_C(x,u)} \iso \phi_{x,u} \cdot e(x,y,\int(  \eta_u ))  \cdot  \comp(x,u,e)^{-1} \, .
\]
The path required to prove  the coherence rule is then obtained using the groupoid laws.
\end{proof}

\begin{corollary} For every $P$-algebra $C$, the type $\isind(C)$ is a mere proposition.
\end{corollary}

\begin{proof} Analogous to that of Corollary~\ref{thm:isbipindisprop}.
\end{proof} 

\subsection*{Homotopy-initial algebras}
Exactly as in the case of bipointed types, the hypothesis that a $P$-algebra $C$ is inductive allows us to show that for any
$P$-algebra $D$, there is a $P$-algebra morphism $f \co C \to D$ which is unique up to a $P$-algebra path, itself is unique up 
to a higher path, which in turn is unique up to a yet higher path, and so on.  As before, we shall characterize this kind of universal property
using the notion of a homotopy-initial $P$-algebra, which we define next.

\begin{definition}\label{def:AlgInit}
Let $C = (C, \sup_C)$ be a $P$-algebra. We say that $C$ is  \emph{homotopy-initial}  if for any $P$-algebra 
$D = (D, \sup_D)$, the type $\Palg(C,D)$ of $P$-algebra morphisms from $C$ to $D$
is contractible, \ie the following type is inhabilited
\[
\isalghinit(C) \defeq
 (\Pi D \co \Palg)  \, \iscontr \big( \Palg(C,D) \big)  \, .
\]  
\end{definition}

We stress again that homotopy-initiality is a purely type-theoretic notion. Also note that, exactly as for 
homotopy-initiality of bipointed types, for a $P$-algebra $C$, the type
$\isalghinit(C)$ is a mere proposition. 
We have the following type-theoretic analogue of Lambek's lemma, which will be used in the proof of Proposition \ref{thm:Whlevel} below.

\begin{lemma}\label{lem:IntLambek} Let $C = (C, \sup_C)$ be a $P$-algebra. 
If $C$ is homotopy-initial, then the structure map $\sup_C \co PC \to C$ is an equivalence.
\end{lemma}

\begin{proof} This is a straightforward translation of the standand category-theoretic proof, but we provide
some details to illustrate where the contractibility condition in the definition of a homotopy-initial algebra is
used. For brevity, let us write $s \co PC \to C$ for the structure map of $C$. 

 We wish to construct a quasi-inverse to $s \co PC \to C$. In order to do so, we use the homotopy-initiality
of $C$. First of all, observe that $PC$ can be made into a $P$-algebra by considering the structure map 
$Ps \co PPC \to PC$. Thus, by the contractibility of the type $\Palg(C, PC)$, there exists a $P$-algebra
morphism $(t, \bar{t}) \co C \to PC$. We represent it as the diagram
\begin{equation*}
\xymatrix@C=1.5cm{
PC \ar[d]_{Pt} \ar[r]^{s} \ar@{}[dr]|{\Downarrow \bar{t}}& C \ar[d]^{t} \\
PPC \ar[r]_{Ps} & PC}
\end{equation*}
Now, the composite $s  t \co C \to C$ and the identity $1_C \co C \to C$ are both $P$-algebra
morphisms and so, by the contractibility of $\Palg(C,C)$, there has to be a path $p \co \Id(s\circ t ,1_C)$. 
Using this fact, we can also show that there is a path $q \co \Id(t  s, 1_{PC})$. Indeed, we have
\[
t \circ s  \iso Ps \circ Pt 
 \iso P(s \circ t) 
 \iso P(1_C) 
  \iso 1_{PC} \, ,
\]
where the first path is given by $\bar{t}$, the second by the pseudo-functoriality of $P$, as in~\eqref{equ:pseudofunP},
the third is the path $p$ constructed above, and the fourth one is given again by the pseudo-functoriality of~$P$,
as in~\eqref{equ:pseudofunP}.
\end{proof}

\begin{proposition} \label{thm:recursiveW}
A $P$-algebra $C = (C, \sup_C)$  is  homotopy-initial if and only if it satisfies the following rules:

\medskip

\begin{enumerate}[(i)]
\item the recursion rule,
\[
\begin{prooftree}
D \co \U \qquad 
x \co A \, ,  u \co B(x) \to D \vdash \sup_D(x,u) \co D 
\justifies
z \co C \vdash \rec(z,\sup_D) \co D
\end{prooftree}
\]
\item the $\beta$-rule,
\[
\begin{prooftree}
D \co \U \qquad 
x \co A \, ,  y \co B(x) \to D \vdash \sup_D(x,y) \co D 
\justifies
x \co A,\, u \co B(x) \to D \vdash
 \beta(x,u ,\sup_D) \co \Id \big( \rec(\sup_C(x,u),\sup_D) \, ,  \sup_D\big(x, (\lambda y \co B(x)) \,  \rec(u y, \sup_D)\big) \big)
\end{prooftree}
\]

\item the $\eta$-rule,
 \smallskip
\[
\begin{prooftree}
\begin{array}{rcl}
 & & D \co \U  \\ 
x \co A,\, u \co B(x) \to D  & \vdash &  \sup_D(x,u) \co D \\ 
z \co C  & \vdash & f(z) \co D \\ 
x \co A, u \co B(x) \to D  & \vdash  & \phi_{x,u} \co \Id( f(\sup_C(x,y)),  \sup_D(x , f u )) 
\end{array}
\justifies
z  \co A \vdash \eta_z \co \Id( f(z) , \rec(z,\sup_D))
\end{prooftree}
\]

\item the $(\beta, \eta)$-coherence rule, 
\[
\begin{prooftree}
\begin{array}{rcl}
& & D \co \U  \\ 
x \co A,\, u \co B(x) \to D &  \vdash & \sup_D(x,u)  \co D \\ 
 z \co C & \vdash & f(z) \co D \\ 
x \co A, u \co B(x) \to D & \vdash & \phi_{x,u} \co \Id( f(\sup_C(x,u)) , \sup_D(x, f \circ u ))
\end{array}
\justifies
x \co A, u \co B(x) \to C \vdash 
\bar{\eta}_{x,u} \co \Id( \beta(x,u,\sup_D) \ct \eta_{\sup_C(x,u)} \,  , 
 \sup_D(x,\int(\eta_u)) \ct \phi_{x,u}   )
\end{prooftree} \smallskip
\]
\end{enumerate}

\end{proposition}

\begin{proof} The rules can be read as follows. The recursion rule says that, given any type $D$ together 
with the function $\sup_D \co PD \to D$, \ie any $P$-algebra, there is a function $r \co C \to D$
defined by letting,  for $z \co C$, $r(z) = \rec(z, \sup_D)$.  The $\beta$-rule implies that we have a homotopy 
$\beta \co \Hot( r \circ \sup_C \, , \sup_D \circ Pr)$ and so, by Proposition~\ref{lem:fibhomeqid}, we get a path $\bar{r}$ fitting in the diagram
\[
\xymatrix@C=1.5cm{
PC \ar[d]_{r} \ar[r]^{\sup_C} \ar@{}[dr]|{\Downarrow \, \bar{r}} & C \ar[d]^{r} \\
PD \ar[r]_{\sup_D} & \; D \, .}
\]
We therefore obtain a $P$-algebra morphism  $(r, \bar{r}) \co C \to D$.  The 
$\eta$-rule says that if $f \co C \to D$ is a $P$-algebra morphism, then there is a homotopy $\eta \co
\Hot( f, r)$.  
And the $(\beta,\eta)$-compatibility rule says that $\eta$ is in fact a $P$-algebra homotopy. Using again 
Proposition~\ref{lem:fibhomeqid}, this shows that there is a path from $(r, \bar{r})$ to $(f, \bar{f})$, thus
proving the contractibility of $\Palg(C,D)$.
\end{proof}

\begin{remark} As for bipointed types, the special case of the rules in Proposition~\ref{thm:recursiveW} 
obtained by considering $C = D$ and $f = 1_C$ provides some explanation for the terminology used to denote them. 
By the recursion rule, we obtain a function $r \co C \to C$ defined by $r = (\lambda z \co C) \rec(z, \sup_C)$. 
The $\beta$-rule gives a homotopy with components $\beta_{x,u} \co \Id( r (\sup_C(x,u), \sup_C( x, r u)$, 
the $\eta$-rule gives a homotopy with components $\eta_z \co \Id( z, r(z))$ and, finally, the $(\beta, \eta)$-coherence rule, gives us a homotopy with components fitting in the diagram
\[
\xymatrix@C=1.5cm{
\sup_C(x,u) \ar[r]^-{\eta_{x,u}} 
\ar@/_1pc/[dr]_{\sup_C(x, \int(\eta_u)) \qquad}
 \ar@{}[dr]|{\qquad \overset{\bar{\eta}_{x,u}}{\Rightarrow}} &
  r(\sup(x,u)) \ar[d]^{\beta_{x,u}} \\ 
 & \sup_C(x, ru) \, .}
 \]
 \end{remark}

We can now state and prove our main result. 

\begin{theorem}\label{thm:WMain} A $P$-algebra is inductive if and only if
it is homotopy-initial, \ie the type
\[ 
(\Pi C \co \Palg) \big( \isalgind(C) \leftrightarrow \isalghinit(C) \big)
\]
is inhabited. 
\end{theorem}

\begin{proof}
Let $C = (C, \sup_C)$ be an inductive $P$-algebra. We wish to show that it is homotopy initial.
For this, it suffices to observe that the rules in Proposition~\ref{thm:recursiveW} characterizing
homotopy-initial algebras are a special case of those given in Proposition~\ref{thm:palgindrec} and Proposition~\ref{lem:Wetaind}.

For the converse, we proceed as in the proof of
Theorem~\ref{thm:bipointedmain}. 
Let $E = (E, e)$ be a fibered algebra over $C$. We need to show that there
exists a $P$-algebra section $(s, \bar{s})$, where $s \co (\Pi x \co C) E(x)$ and 
\[ 
\bar{s} \co (\Pi x \co A)(\Pi u \co B(x) \to C) \Id \big( s(\sup_C(x,u)), e(x, u, s u) \big)
\]
We consider the $P$-algebra $(E', \sup_{E'})$ associated to $E$. Recall that $E' \defeq
(\Sigma z \co C) E(z)$ and $\sup_{E'} \co PE' \to E'$ is defined so that, for $x \co A$ and $u \co B(x) \to E'$, 
we have
\[
\sup_{E'}(x,u) = \big( \sup_C(x,\pi_1  u) \, , e (x,\pi_1  u, \pi_2  u )\big) \, .
\]
In this way, the first projection $\pi_1 \co E' \to C$  is an algebra morphism, represented by the diagram
 \[
 \xymatrix@C=1.5cm{
 PE' \ar[d]_-{P \pi_1} \ar[r]^{\sup_{E'}} \ar@{}[dr]|{\Downarrow \,  \overline{\pi}_1} & E' \ar[d]^{\pi_1} \\
 PC \ar[r]_{\sup_C} & \; C \, .}
 \]
By the homotopy-initiality of $C$, there exists an algebra morphism $(f, \bar{f}) \co (C, \sup_C)  \to (E', \sup_{E'})$,
which we represent with the diagram
\[
\xymatrix@C=1.5cm{
 PC \ar[r]^{\sup_C} \ar[d]_{Pf}  \ar@{}[dr]|{\Downarrow \, \bar{f}} &  C \ar[d]^{f}\\
PE' \ar[r]_{\sup_{E'}}   & \; E' \, .}
\] 
Let $\phi \defeq (\ext \bar{f})$ the homotopy associated to the path $\bar{f}$. 
We write $f_1 \co C \to C$ for the composite $\pi_1 f \co C \to C$, which is a $P$-algebra morphism. The 
path
\[
\xymatrix@C=1.5cm{
PC \ar[d]_{Pf_1} \ar[r]^-{\sup_{C}} \ar@{}[rd]|{\Downarrow \, \overline{f}_1} & C \ar[d]^{f_1} \\ 
PC  \ar[r]_-{\sup_C} & C }
 \]
is given by the pasting diagram
\[
\xymatrix@C=1.5cm{
 PC \ar[r]^{\sup_C} \ar[d]_{Pf}  \ar@{}[dr]|{\Downarrow \, \bar{f}} &  C \ar[d]^{f}\\
PE \ar[r]_{\sup_{E'}}  \ar[d]_{P\pi_1}   \ar@{}[dr]|{\Downarrow \, \bar{\pi_1}} & E' \ar[d]^{\pi_1}  \\
PC \ar[r]_{\sup_C} & \; C  \, .}
\] 
Let $\phi_1 \defeq (\ext \, \bar{f}_1)$ be the homotopy associated to the path $\bar{f_1}$. 
Unfolding the definitions, we have that, for $x \co A,  u \co B(x) \to C$, 
\begin{equation}
\label{equ:smallkey}
(\phi_1)_{x,u} \iso   \pi_1\, \ext^\Sigma \, \phi_{x,u}   \, .
\end{equation}
By the homotopy-initiality of $C$ and Lemma~\ref{IdEqHo}, there exists a $P$-algebra 
homotopy 
\[
(\alpha, \bar{\alpha}) \co \AlgHot(f_1 , 1_C) \, ,
\]
where  $\alpha \co \Hot( f_1 , 1_C)$ and, for $x \co A$, $u \co B(x) \to C$, the path
$\bar{\alpha}_{x,u}$ fits into the diagram
\begin{equation}
\label{equ:alphabar}
{\vcenter{\hbox{\xymatrix@C=2cm{
f_1 \sup_C(x,u)
 \ar[r]^{(\phi_1)_{x,u}} \ar[d]_{\alpha_{\sup_C(x,u)}}  
\ar@{}[dr]|{\Downarrow \, \bar{\alpha}_{x,u}} & \sup_C(x, f_1 u) \ar[d]^{\sup_C(x, \int(\alpha_u))}  \\ 
\sup_C(x,u) \ar[r]_{\bar{1}_{x,u}} & \sup_C(x,u) \, . }}}}
\end{equation}
Here, $\int(\alpha_u )\co \Id( f_1  u \, ,  u)$ is the path associated to the homotopy
$(\lambda y \co B(x)) \alpha_{u y} \co \Hot( f_1 u , u )$. 
We define the required section $s \co (\Pi z \co C) E(z)$ so that, for $z \co C$ we have
\begin{equation*}
s z \defeq (\alpha_z)_{!} \big(  f_2 z \big) \, ,
\end{equation*}
where  $(\alpha_z)_{!} : E(f_1 z) \to E(z)$ is a transport map associated to the path
$\alpha_z \co \Id( f_1 z,  z)$. It now remains to  define, for each $x \co A$ and $u \co B(x)\to C$, 
a path
\[
\bar{s}(x,u) \co \Id\big(  s(\sup_C(x,u)) \, ,  e_s(x,u)  \big) \, ,
\] 
where $e_s$ is defined using the formula in~\eqref{equ:ef}. 
Unfolding the definitions,  our goal is to show that
\begin{equation}\label{eq:proof:thm:WMain:needpath}
(\alpha_{\sup_C(x,u)})_{!} (  f_2 \, \sup_C(x,u) ) \iso 
e(x, u, (\lambda y \co B(x)) (\alpha_{u y})_{!} ( f_2 u y)   )   \, .
\end{equation}
Our goal will follow once we show the following:
\begin{align*}
\textit{Claim 1.} &  \qquad   \alpha_{\sup_C(x,u)} \iso \sup_C(x, \int(\alpha_u))  \ct (\phi_1)_{x,u}   \\ 
\textit{Claim 2.} &  \qquad  ((\phi_1)_{x,u})_{!} (f_2 \, \sup_C(x,u)) \iso e(x,  f_1 u, f_2 u)    \, . \\ 
\textit{Claim 3.} &  \qquad  (\sup_C(x, \int(\alpha_u)))_{!} \, e(x, f_1 u, f_2 u)  \iso 
e\big(x, u, (\lambda y \co B(x)) (\alpha_{uy})_{!}(f_2 u y) \big) \, . 
\end{align*}
Inded, the required propositional equality in~\eqref{eq:proof:thm:WMain:needpath} can then be obtained as follows:
 \begin{alignat*}{2} 
 (\alpha_{\sup_C(x,u)})_{!} ( f_2 \, \sup_C(x,u) ) 
 &  \iso  
 (\sup_C(x, \int(\alpha_u)))_{!} \;  ((\phi_1)_{x,u})_{!}  (f_2 \, \sup_C(x,u) )  
 & \quad &  \text{(by Claim 1)} \\ 
& \iso  
 (\sup_C(x, \int(\alpha_u)))_{!} \;   e(x, f_1 u, f_2 u)  
 & \quad &  \text{(by Claim 2)} \\
&  \iso e(x, u, (\lambda y \co B(x)) (\alpha_{uy})_{!}( f_2 u y))  
     & \quad  & \text{(by Claim 3).} 
  \end{alignat*} 
We conclude by proving the auxiliary claims stated above.

\begin{proof}[Proof of Claim 1.]  This follows by the path in the diagram in~\eqref{equ:alphabar}. \noqed
\end{proof}

\begin{proof}[Proof of Claim 2] Recall that the homotopy $\phi$ has components 
 \[
\phi_{x,u} \co \Id \big( f  \,  \sup_C(x,u) \, ,  \sup_{E'}(x, fu) \big)
\]
Thus, by the characterization of paths in $\Sigma$-types, we have
\[
p  \co \Id( f_1 \sup_C(x,u) , \sup_C(x, f_1 u)) \, , \quad
q \co  
\Id\big( p_{!} ( f_2 \, \sup_C(x,u))    \, ,     e(x, f_1 u, f_2 u) \big) \, ,
 \]   
 where $p \defeq \pi_1 \, \ext^\Sigma \phi_{x,u}$ and $q \defeq \pi_2 \, \ext^\Sigma \, \phi_{x,u}$. 
 The claim now follows by~\eqref{equ:smallkey}. 
 \noqed
\end{proof}

\vspace{-1ex}
 
 \begin{proof}[Proof of Claim 3] 
Observe that for all $a \co A$, $p \co \Id_{B(a) \to C}(t_1, t_2)$  and 
$v \co (\Pi y \co B(a))E(t_1y)$, we have
\[ 
(\sup_C(a,p))_{!}\, e(a,t_1,v)  \iso 
e \big( a, t_2,  (\lambda y \co B(a)) ( (\ext \, p)_y )_{!}\, v y \big) \big)   \, . 
\]
by $\Id$-elimination. If we apply this to $x \co A$, $\int(\alpha_u) \co \Id( f_1 u, u)$, and $f_2 u \co
(\Pi y \co B(x)) E(f_1 u y)$, we get
\[
\big( \sup_C(x,\int(\alpha_u)) \big)_{!} \; e(x, f_1 u , f_2 u) \iso 
e\big(x, u , (\lambda y \co B(x)) \big( (\alpha_{u y} \big)_{!} \;  f_2 u y  \big) \big) 
 \, ,
\]
as required.   \end{proof}  \noqed
\end{proof}

\vspace{-0.8cm}

\begin{corollary} For every $P$-algebra $C$, there is an equivalence $\isalgind(C) \simeq \isalghinit(C)$.
\end{corollary}

\begin{proof} Theorem~\ref{thm:WMain} gives us a logical equivalence, but both types are mere propositions.
\end{proof} 

Below, when we assume the rules for $W$-types (as in Table~\ref{tab:wrules}), we always write $W$ for $(\W x \co A) B(x)$.

\begin{corollary}
\label{lem:WInitInt} Assuming the rules for $W$-types, for a $P$-algebra
$C$ the following conditions are
equivalent:
\begin{enumerate}[(i)]
\item $C$ is inductive,
\item $C$ is homotopy initial,
\item $C$ is equivalent to $W$ as a $P$-algebra.
\end{enumerate}
In particular, the type $W$ is a homotopy-initial $P$-algebra. \qed
\end{corollary}

Corollary~\ref{lem:WInitInt} provides the analogue in our setting of the characterization of W-types as a strict initial algebra in extensional type theory. It makes precise the informal idea that, in intensional type theory, W-types are a kind of initial algebra in the weak $(\infty, 1)$-category of types, functions, paths and higher paths.  

\begin{lemma} \label{lem:suppath}
Assuming the rules for $W$-types, 
for all $a_1,a_2 \co A$, $t_1 \co B(x_1) \to W$, $t_2 \co B(x_2) \to W$, there is an equivalence of types
\[ 
\Id_W( \sup_W(a_1,t_1),  \sup_W(a_2,t_2)  ) \simeq  \Id_{PW} \big( (a_1,t_1), (a_2,t_2) \big) \, . 
\]
\end{lemma}

\begin{proof}
By Lemma~\ref{lem:IntLambek} and Corollary~\ref{lem:WInitInt}, $\sup_W \co PW \to W$ is an equivalence.
\end{proof}

We  remark that $W$-types  preserve homotopy levels, in the following sense (see \cite{Danielsson:Wtypes}).

\begin{proposition}\label{thm:Whlevel}
Assuming the rules for $W$-types, if $A$ has h-level $n+1$, then so does the $W$-type $(\W x:A)B(x)$.
\end{proposition}

\begin{proof}
We need to show that for all $w, w' \co W$ the type $\Id_W(w, w')$ has h-level $n$. We do so applying the elimination
rule for W-types on $w \co W$.  So, let $x \co A, u \co B(x) \to W$ and assume the induction hypothesis 
\begin{itemize}
\item[$(\ast)$] for every $y \co B(x)$, for every $w' \co W$,  the type $\Id_W( uy ,w')$ has h-level $n$, 
\end{itemize}
and show that for every $w' \co W$ the type $\Id(\sup_W(x,u), w')$ has h-level $n$. We apply again the elimination rule for W-types. So, let   $x' \co A$, $u' \co B(x') \to W$ and assume the induction hypothesis (which we do not spell out since we will not need it) and show that $\Id( \sup_W(x,u) , \sup_W(x',u'))$ has h-level $n$. We have
\begin{align*} 
\Id_W(\sup_W(x,u), \sup_W(x',u'))
& \simeq \Id_{PW}((x,u) , (x',u')) \\
& \simeq (\Sigma p \co \Id_A(x, x')) \, \Id_{B(x) \to W}( u , p^*(u') )   \\
& \simeq (\Sigma p \co \Id_A(x,x')) \,  \Id\big(u , (\lambda y \co B(x)) \,  u'( p_{!} \, y)\big) \\
& \simeq (\Sigma p \co \Id_A(x,x')) (\Pi y \co B(x)) \, \Id_W \big( uy  , u'(p_{!} \, y)\big) \, .
\end{align*}
Here, the first equivalence follows by Lemma~\ref{lem:suppath} and the other equivalences follow by standard
properties of the transport functions. Since $A$ has h-level $n+1$ by assumption, we have 
that~$\Id_A(x,x')$ has h-level $n$. Also, for any $p \co \Id_A(x,x')$ and $u \co B(x) \to W$, the 
type~$\Id_W(uy ,  u'(p_{!} \, y))$ has h-level $n$ by the 
induction hypothesis in~$(\ast)$. The claim follows by recalling that the h-levels are closed under arbitrary dependent products and under dependent sums over types of the same h-level. 
\end{proof}

We note that the h-level of $(\W x \co A) B(x)$ does not depend on that of $B(x)$. Furthermore, assuming that we have a unit type $1$, the lemma is no longer true if $n+1$ is replaced by $n$, as the following example illustrates: if $A \defeq 1$ and~$B(x) \defeq 1$, then $(\W x \co A) B(x) \simeq 0$, which is not contractible.

\subsection*{Univalence for algebras} \label{sec:univalencealgebras}
We conclude the paper with some applications of the univalence axiom. The first is that, just as for bipointed 
types, a form of univalence holds also for $P$-algebras, as the next theorem makes precise.

\begin{theorem}\label{thm:Punivalence} Assuming the univalence axiom, 
 the canonical function
\[ 
\ext^\Palg_{C,D} \co \Id \big(C,D\big) \to  \AlgEquiv(C,D) 
\]
is an equivalence for every pair of $P$-algebras $C$ and $D$.
\end{theorem}

\begin{proof} 
Let $C = (C,\sup_C)$ and $D= (D,\sup_D)$ be $P$-algebras. By the characterization of paths in $\Sigma$-types, 
$\Id \big( (C,\sup_C) ,  (D,\sup_D) \big)$ can be expressed as the type
\[
(\Sigma p \co \Id(C, D)) \,  \Id \big( \sup_C ,  p^*(\sup_D)  \big) \, .
\]
By path induction on $p$ and the characterization of paths in $\Pi$-types, this type is equivalent to
\[  
(\Sigma p \co \Id(C,D))
(\Pi x \in A) 
(\Pi u \co B(x) \to W)
\Id \big(  (\ext \, p)( \sup_C(x,u)),  \sup_D(x, (\ext \, p ) \circ u) \big) \, , 
\]
where $\ext \co \Id(C,D) \to \mathsf{Equiv}(C,D)$ is the canonical extension function for the identity types of elements
of~$\U$, asserted to be an equivalence by the univalence axiom. Hence, the above type is equivalent to
\[
(\Sigma f \co \mathsf{Equiv}(C,D)) 
(\Pi x  \co A) 
(\Pi y \co B(x) \to W) \, 
\Id \big( f (\sup_C(x,u))  , \sup_D (x, f u) \big) \, .
\]
After rearranging, we get
\[
(\Sigma f \co \Palg ( (C,\sup_C),  (D,\sup_D) ) \, \isequiv( f ) \, .
\]
By Proposition~\ref{WAlgSpace}, this type is equivalent to $\AlgEquiv \big( (C,\sup_C),  (D,\sup_D)\big)$, as desired. Finally, it is not hard to see that the composition of the above equivalences yields, up to a homotopy, the canonical function $\ext$ which is therefore an equivalence, as required.
\end{proof} 

The following corollary, still obtained under the assumption of the univalence axiom, shows that
homotopy-initial algebras are unique up to a unique path.

\begin{corollary}  Assuming the univalence axiom,
homotopy-initial $P$-algebras are unique up to a  contractible type of paths, i.e. the type
\[ 
(\Pi C \co \Palg) (\Pi D \co \Palg) \big( \isalghinit(C) \times \isalghinit(D)  \to 
\iscontr(\Id(C,D)) \big) \, .
\] 
is inhabited.
\end{corollary}

\begin{proof} This is an immediate consequence of Theorem~\ref{thm:Punivalence} and Proposition~\ref{WHInitIso}. 
\end{proof}

\section*{Acknowledgements}

We would like to thank Vladimir Voevodsky and Michael Warren for helpful discussions
on the subject of this paper. In particular, the former suggested a simplification of the 
proof that an inductive $P$-algebra is homotopy-initial. \smallskip
 
 \noindent
 \emph{Acknowledgements for Steve Awodey.}  Steve Awodey gratefully acknowledges the support of the National Science Foundation, Grant DMS-1001191 and the Air Force OSR, Grant Number 11NL035 and MURI Grant Number FA9550-15-1-0053. \smallskip
 
 \noindent
 \emph{Acknowledgements for Nicola Gambino.}   
 \begin{enumerate}[(i)]
 \item This work was supported by the National Science Foundation 
under agreement No.\ DMS-0635607. 
\item This material is based on research sponsored by the Air Force Research Laboratory, under agreement 
number FA8655-13-1-3038. \item This research was supported by a grant from the John Templeton Foundation.
\item This research was supported by an EPSRC grant (EP/M01729X/1). 
\end{enumerate}

\smallskip

\noindent
\emph{Acknowledgements for Kristina Sojakova.} The support of CyLab at Carnegie
Mellon under grants DAAD19-02-1-0389 and W911NF-09-1-0273 from the Army
Research Office is gratefully acknowledged.

\bibliographystyle{plain}

\bibliography{references}

\end{document}

%% file: prooftree.tex
\newdimen\proofrulebreadth \proofrulebreadth=.05em
\newdimen\proofdotseparation \proofdotseparation=1.25ex
\newdimen\proofrulebaseline \proofrulebaseline=2ex
\newcount\proofdotnumber \proofdotnumber=3
\let\then\relax
\def\hfi{\hskip0pt plus.0001fil}
\mathchardef\squigto="3A3B
\newif\ifinsideprooftree\insideprooftreefalse
\newif\ifonleftofproofrule\onleftofproofrulefalse
\newif\ifproofdots\proofdotsfalse
\newif\ifdoubleproof\doubleprooffalse
\let\wereinproofbit\relax
\newdimen\shortenproofleft
\newdimen\shortenproofright
\newdimen\proofbelowshift
\newbox\proofabove
\newbox\proofbelow
\newbox\proofrulename
\def\shiftproofbelow{\let\next\relax\afterassignment\setshiftproofbelow\dimen0 }
\def\shiftproofbelowneg{\def\next{\multiply\dimen0 by-1 }%
\afterassignment\setshiftproofbelow\dimen0 }
\def\setshiftproofbelow{\next\proofbelowshift=\dimen0 }
\def\setproofrulebreadth{\proofrulebreadth}

\def\prooftree{
\ifnum  \lastpenalty=1
\then   \unpenalty
\else   \onleftofproofrulefalse
\fi
\ifonleftofproofrule
\else   \ifinsideprooftree
        \then   \hskip.5em plus1fil
        \fi
\fi
\bgroup
\setbox\proofbelow=\hbox{}\setbox\proofrulename=\hbox{}%
\let\justifies\proofover\let\leadsto\proofoverdots\let\Justifies\proofoverdbl
\let\using\proofusing\let\[\prooftree
\ifinsideprooftree\let\]\endprooftree\fi
\proofdotsfalse\doubleprooffalse
\let\thickness\setproofrulebreadth
\let\shiftright\shiftproofbelow \let\shift\shiftproofbelow
\let\shiftleft\shiftproofbelowneg
\let\ifwasinsideprooftree\ifinsideprooftree
\insideprooftreetrue
\setbox\proofabove=\hbox\bgroup$\displaystyle 
\let\wereinproofbit\prooftree
\shortenproofleft=0pt \shortenproofright=0pt \proofbelowshift=0pt
\onleftofproofruletrue\penalty1
}

\def\eproofbit{
\ifx    \wereinproofbit\prooftree
\then   \ifcase \lastpenalty
        \then   \shortenproofright=0pt  
        \or     \unpenalty\hfil         
        \or     \unpenalty\unskip       
        \else   \shortenproofright=0pt  
        \fi
\fi
\global\dimen0=\shortenproofleft
\global\dimen1=\shortenproofright
\global\dimen2=\proofrulebreadth
\global\dimen3=\proofbelowshift
\global\dimen4=\proofdotseparation
\global\count255=\proofdotnumber
$\egroup  
\shortenproofleft=\dimen0
\shortenproofright=\dimen1
\proofrulebreadth=\dimen2
\proofbelowshift=\dimen3
\proofdotseparation=\dimen4
\proofdotnumber=\count255
}

\def\proofover{
\eproofbit 
\setbox\proofbelow=\hbox\bgroup 
\let\wereinproofbit\proofover
$\displaystyle
}%
\def\proofoverdbl{
\eproofbit 
\doubleprooftrue
\setbox\proofbelow=\hbox\bgroup 
\let\wereinproofbit\proofoverdbl
$\displaystyle
}%
\def\proofoverdots{
\eproofbit 
\proofdotstrue
\setbox\proofbelow=\hbox\bgroup 
\let\wereinproofbit\proofoverdots
$\displaystyle
}%
\def\proofusing{
\eproofbit 
\setbox\proofrulename=\hbox\bgroup 
\let\wereinproofbit\proofusing
\kern0.3em$
}

\def\endprooftree{
\eproofbit 
  \dimen5 =0pt
\dimen0=\wd\proofabove \advance\dimen0-\shortenproofleft
\advance\dimen0-\shortenproofright
\dimen1=.5\dimen0 \advance\dimen1-.5\wd\proofbelow
\dimen4=\dimen1
\advance\dimen1\proofbelowshift \advance\dimen4-\proofbelowshift
\ifdim  \dimen1<0pt
\then   \advance\shortenproofleft\dimen1
        \advance\dimen0-\dimen1
        \dimen1=0pt
        \ifdim  \shortenproofleft<0pt
        \then   \setbox\proofabove=\hbox{%
                        \kern-\shortenproofleft\unhbox\proofabove}%
                \shortenproofleft=0pt
        \fi
\fi
\ifdim  \dimen4<0pt
\then   \advance\shortenproofright\dimen4
        \advance\dimen0-\dimen4
        \dimen4=0pt
\fi
\ifdim  \shortenproofright<\wd\proofrulename
\then   \shortenproofright=\wd\proofrulename
\fi
\dimen2=\shortenproofleft \advance\dimen2 by\dimen1
\dimen3=\shortenproofright\advance\dimen3 by\dimen4
\ifproofdots
\then
        \dimen6=\shortenproofleft \advance\dimen6 .5\dimen0
        \setbox1=\vbox to\proofdotseparation{\vss\hbox{$\cdot$}\vss}%
        \setbox0=\hbox{%
                \advance\dimen6-.5\wd1
                \kern\dimen6
                $\vcenter to\proofdotnumber\proofdotseparation
                        {\leaders\box1\vfill}$%
                \unhbox\proofrulename}%
\else   \dimen6=\fontdimen22\the\textfont2 
        \dimen7=\dimen6
        \advance\dimen6by.5\proofrulebreadth
        \advance\dimen7by-.5\proofrulebreadth
        \setbox0=\hbox{%
                \kern\shortenproofleft
                \ifdoubleproof
                \then   \hbox to\dimen0{%
                        $\mathsurround0pt\mathord=\mkern-6mu%
                        \cleaders\hbox{$\mkern-2mu=\mkern-2mu$}\hfill
                        \mkern-6mu\mathord=$}%
                \else   \vrule height\dimen6 depth-\dimen7 width\dimen0
                \fi
                \unhbox\proofrulename}%
        \ht0=\dimen6 \dp0=-\dimen7
\fi
\let\doll\relax
\ifwasinsideprooftree
\then   \let\VBOX\vbox
\else   \ifmmode\else$\let\doll=$\fi
        \let\VBOX\vcenter
\fi
\VBOX   {\baselineskip\proofrulebaseline \lineskip.2ex
        \expandafter\lineskiplimit\ifproofdots0ex\else-0.6ex\fi
        \hbox   spread\dimen5   {\hfi\unhbox\proofabove\hfi}%
        \hbox{\box0}%
        \hbox   {\kern\dimen2 \box\proofbelow}}\doll%
\global\dimen2=\dimen2
\global\dimen3=\dimen3
\egroup 
\ifonleftofproofrule
\then   \shortenproofleft=\dimen2
\fi
\shortenproofright=\dimen3
\onleftofproofrulefalse
\ifinsideprooftree
\then   \hskip.5em plus 1fil \penalty2
\fi
}
